\documentclass[12pt]{amsart}
\usepackage{amsmath,amssymb,latexsym,xcolor,soul,cite,amsthm,enumitem,graphicx,tikz,mathtools,microtype}
\usepackage{verbatim}
\usepackage[left=2.3cm,right=2.3cm,top=2.8cm,bottom=2.8cm]{geometry}
\usepackage[colorlinks=true, urlcolor=blue, citecolor=red, linkcolor=blue, pdffitwindow=true, linktocpage, pdfpagelabels, bookmarksnumbered, bookmarksopen]{hyperref}
\usepackage[hyperpageref]{backref} 
\usepackage[english]{babel}
\allowdisplaybreaks

\renewcommand{\epsilon}{\varepsilon}

\numberwithin{equation}{section}
\newtheorem{theorem}{Theorem}[section]
\newtheorem{proposition}[theorem]{Proposition}
\newtheorem{lemma}[theorem]{Lemma}
\newtheorem{remark}[theorem]{Remark}

\newtheorem{corollary}[theorem]{Corollary}
\newtheorem{definition}[theorem]{Definition}
\theoremstyle{definition}

\newcommand{\eps}{\varepsilon}

\newcommand{\N}{{\mathbb N}}

\newcommand{\R}{\mathbb{R}}

\def\Xiint#1{\mathchoice
   {\XXiint\displaystyle\textstyle{#1}}%
   {\XXiint\textstyle\scriptstyle{#1}}%
   {\XXiint\scriptstyle\scriptscriptstyle{#1}}%
   {\XXiint\scriptscriptstyle\scriptscriptstyle{#1}}%
   \!\iint}
\def\XXiint#1#2#3{{\setbox0=\hbox{$#1{#2#3}{\iint}$}
     \vcenter{\hbox{$#2#3$}}\kern-.5\wd0}}
\def\longminus{\raisebox{-1ex}{\rotatebox[origin=c]{0}{${-}\mkern-3.5mu{-}$}}}
\def\dashiint{\Xiint\longminus}

\def\XXint#1#2#3{{\setbox0=\hbox{$#1{#2#3}{\int}$ }
\vcenter{\hbox{$#2#3$ }}\kern-.57\wd0}}

\title[Speed of propagation for anisotropic parabolic PDE]{Anisotropic Sobolev embeddings and the speed of propagation for parabolic equations}

\author[F.\ G.\ D\"uzg\"un, S.\ Mosconi \& V. Vespri]{Fatma Gamze D\"uzg\"un, Sunra Mosconi$^{*}$ and Vincenzo Vespri}

\address[F.\ G.\ D\"uzg\"un]{Department of Mathematics, 
\newline\indent
Hacettepe University, 06800, Beytepe, Ankara, Turkey}
\email{gamzeduz@hacettepe.edu.tr }

\address[S.\ Mosconi]{Dipartimento di Matematica e Informatica,
\newline\indent
Universit\`a degli Studi di Catania,
Viale A.\ Doria 6, 95125 Catania, Italy}
\email{mosconi@dmi.unict.it}

\address[V. Vespri]{Dipartimento di Matematica e Informatica ``U. Dini'', 
\newline\indent
Universit\`a di Firenze, Viale Morgagni 67/A, 50134 Firenze, Italy}
\email{vespri@math.unifi.it}


\begin{document}

\begin{abstract}
We consider a quasilinear parabolic  Cauchy problem with spatial anisotropy of orthotropic type and study the spatial localization of solutions. Assuming the initial datum is localized with respect to a coordinate having slow diffusion rate, we bound the corresponding directional velocity of the support along the flow. The expansion rate is shown to be optimal for large times. 
\end{abstract}

\maketitle

\begin{center}
	\begin{minipage}{8.3cm}
		\small
		\tableofcontents
	\end{minipage}
\end{center}

\section{Introduction}

This paper deals with some properties of weak solutions to the Cauchy problem
\begin{equation}
\label{CaP}
\begin{cases}
u_{t}-{\rm div} A(x, u, Du)=0&\text{ in $S_{T}=\R^{N}\times \ ]0, T[$},\\
u(x, 0)=u_{0}.
\end{cases}
\end{equation}
The field  $A:S_{T}\times \R\times \R^{N}\to \R^{N}$ is only measurable and {\em anisotropic}, i.e.
\begin{equation}
\label{gcond}
\begin{cases}
A_{i}(x, s, z)\, z_{i}\geq \Lambda^{-1}|z_{i}|^{p_{i}}\\
|A_{i}(x, s, z)|\leq \Lambda|z_{i}|^{p_{i}-1}
\end{cases}
\end{equation}
for some $\Lambda>0$ and a suitable choice of $p_{i}> 1$, $i=1,\dots, N$. 
In the  case $p_{i}\equiv p$, $A_{i}(x, s, z)=|z_{i}|^{p-2}z_{i}$, the equation in \eqref{CaP} is known as the parabolic orthotropic $p$-Laplacian equation and, while its principal part is homogeneous, it fails to be isotropic.

Many materials, such as liquid crystals, wood  or earth's crust usually present different diffusion rates along different directions. Moreover, in most of the physical phenomena involved in such media, finite speed of propagation of disturbances is a much more reasonable assumption than the usual infinite-speed one implied by linear equations. This effect can either be caused by additional absorption terms in the model equation or by the intrinsic  diffusion rate of the medium. We are interested in this latter situation which, in the framework assumed here, consists in studying \eqref{CaP} when some of the $p_{i}$'s are greater than $2$.

Let us discuss some features of equation \eqref{CaP} from the mathematical point of view.

\vskip4pt

{\em -- Regularity}\\ 
The anisotropy prescribed by \eqref{gcond} falls into the wider class of problems with {\em non-standard} growth condition. Even in the stationary case, a reasonable regularity theory for such equations requires a bound on the sparseness of the powers $p_{i}$. For elliptic equations driven by a principal part of the form \eqref{gcond}, counterexamples to boundedness are given in the seminal papers \cite{G, M, H}. Sufficient conditions for local boundedness are found in \cite{FS, BMS, DBGV16} and read
\[
\bar p<N,\qquad \max\{p_{1},\dots, p_{N}\}<\bar p^{*},
\]
where
\begin{equation}
\label{d}
\frac{1}{\bar p}=\frac{1}{N}\sum_{i=1}^{N}\frac{1}{p_{i}},\qquad \bar p^{*}=\frac{N\bar p}{N-\bar p}
\end{equation}
while in the limit case $\max\{p_{1}, \dots, p_{N}\}=\bar p^{*}$ it suffices to assume that $u\in L^{\bar p^{*}}_{\rm loc}(\R^{N})$ (which does not follow from the anisotropic Sobolev embedding, see \cite{KK}). With stronger conditions on the $p_{i}$'s and more smoothness assumptions on $A$, H\"older and Lipschitz regularity can be obtained and the literature on the elliptic case is huge. We refer to \cite[Section 6]{Min} and \cite{EMM} for a complete survey on the subject and related bibliography. 

Regarding the parabolic degenerate (or singular) case of \eqref{CaP} much less is known. Indeed, the regularity theory going from linear to degenerate/singular equations in the parabolic framework is much more subtle than in the elliptic one. As far as we know, continuity of solutions to \eqref{CaP} is known only when the principal part is homogeneous, i.e. $p_{i}\equiv p$, thanks to the work of DiBenedetto and collaborators, see \cite{DBGVbook}. Nevertheless, some $L^{\infty}$-theory has been developed in the fully anisotropic case, see \cite{MX, DT07} and the book \cite{AS}.  

\vskip4pt

{\em -- Existence/Uniqueness}\\  In general, even for the heat equation, the Cauchy problem suffers heavy non-uniqueness phenomena, even for smooth and compactly supported initial data. This holds  true also for the general isotropic model equation 
\begin{equation}
\label{mod}
\begin{cases}
u_{t}={\rm div}(|Du|^{p-2}Du)&\text{in $S_{T}$},\\
u(x, 0)=u_{0},
\end{cases}
\end{equation}
and examples can be found either in \cite{BV} or by coupling the results in \cite{V} (see also \cite{Hu}) for the porous media equation with \cite{V2}. In appendix \ref{sec:appA} we will exhibit self-similar examples {\em \'a la} Tikhonov \cite{Ty} (i.e., solutions $u\neq 0$ for $u_{0}\equiv 0$) through a different and direct approach.\\
The main point of the theory is then to identify natural function classes ensuring unique solvability of the Cauchy problem. For linear parabolic equations this is a classical theme initiated by Widder \cite{W} and developed by Aronsson \cite{A}, see the Introduction to \cite[Ch XI]{DBbook} for a discussion. Roughly speaking, well-posedness holds for the class of functions (and initial data) obeying a suitable growth condition at infinity. Moreover, the growth condition is called optimal if any non-negative solution on $S_{T}$ {\em a-priori} satisfies it. 

The optimal well-posedness class was determined for \eqref{mod} in \cite{DBH, DBH2}, following a number of related works on the porous medium equation. In the model case of the anisotropic parabolic equation a growth condition ensuring existence and uniqueness has been studied in \cite{DT07, DT12}, however its optimality is still missing up to now. Indeed, the proofs of optimality usually involve some form of Harnack inequality in turn implying the prescribed growth, but in the anisotropic case no Harnack inequality is in general known.

A much smaller class where well-posedness holds true for \eqref{mod} is the one of $L^{p}(S_{T})$ solutions. 
Lions' method \cite{L} indeed gives existence and uniqueness for the Cauchy problem with $L^{2}(\R^{N})$ initial data, and eventually the method can be extended to problem \eqref{CaP}  assuming monotonicity of $z\mapsto A(x, s, z)$ and $L^{1}(\R^{N})$  initial data (or even $u_{0}$ being a finite measure). In this respect see \cite{DT12} and the book \cite{AS}. 
However, at the level of generality dictated by the s\^ole condition \eqref{gcond} in \eqref{CaP}, no existence and/or uniqueness theory is available, whatever {\em a-priori} summability condition one imposes on the solution. 

\vskip4pt

{\em -- Description of the result and related problems}\\
We are interested in finding an optimal bounds for the directional speeds of propagation of the support of $u_{0}$ under the flow \eqref{CaP}. Due to the changes in the diffusion coefficients with respect to various coordinates, we expect different velocities in each direction.  It is well-known that, when the principal part is homogeneous, i.e.  $p_{i}\equiv p$, finite energy solutions of \eqref{CaP} with initial condition $u_{0}\in C^{\infty}_{c}(\R^{N})$ preserve compactness of the support along the motion whenever $p>2$ (slow-diffusion rate).  In this case, an optimal bound from above to the speed of ${\rm supp}(u(\cdot, t))$ has been determined in \cite{AT}, \cite{BRVV} for a wide class of second order nonlinear parabolic equations, in \cite{B} for higher order ones, while \cite{TV} deals with suitable parabolic systems. A typical result in this setting reads as follows: if $u\in L^{p}(S_{T})$ solves \eqref{mod} for some $p>2$ and ${\rm supp}(u_{0})\subseteq B_{R_{0}}$, then for any $t\in \ ]0, T]$
\begin{equation}
\label{jap}
 {\rm supp}(u(\cdot, t))\subseteq B_{2R_{0}+R(t)},\qquad 
R(t)= Ct^{\frac{1}{N(p-2)+p}}\|u_{0}\|_{1}^{\frac{p-2}{N(p-2)+p}},
\end{equation}
where here and in the following we denote by $\|\ \|_{p}$ the $L^{p}(\R^{N})$ norm, $p\in [1, \infty]$.
In the anisotropic case, the model equation presents {\em locally} in space the same {\em qualitative} finite speed of propagation  in each in direction $x_{i}$ of slow diffusion rate (i.e., for which $p_{i}>2$). We refer to \cite{AS} and the literature therein for these kinds of results. What seems to be missing is a {\em global}, {\em quantitative} estimate of the aforementioned velocities, as prescribed e.g. by \eqref{jap} in the homogeneous principal part scenario.

Since we  will deal with local solutions to \eqref{CaP}, the well-posedeness issues described above suggest that we cannot  expect finite speed of propagation to hold for generic local solutions, even for the parabolic $p$-Laplace equation \eqref{mod}. Indeed, in appendix \ref{sec:appA} (see also Remark \ref{branch}), self-similar solutions  are constructed showing that finite speed of propagation fails at the global level. The main step consists in proving a non-uniqueness result for a merely H\"older continuous system of ODE with suitable initial data.  In order to achieve this, we employ a sub-supersolution method which we think will be of some use for other problems, avoiding the full dynamical system analysis performed, e.g., in  \cite{Hu} or \cite{BV}.\\
Despite the failure of finite speed of propagation in general, we will see that any local solution $u$ of \eqref{CaP} actually possesses a so-called {\em branch} $\tilde u$ (i.e., a solution of \eqref{CaP} coinciding with $u$ on ${\rm supp}(\tilde u)$, see Remark \ref{branch} for some examples), exhibiting finite speed of propagation in each direction of slow diffusion. Finally, we will obtain a quantitative bound from above on the velocity, showing its optimality in a wide range of anisotropies. 

For the purpose of this introduction we state our theorem in the case when all the $p_{i}$  are greater than $2$. We will also consider the case when some (but not all) of the $p_{i}$ are less than or equal than $2$, but the statement is less transparent, and we refer to the last section (namely, Theorem \ref{mainteo} therein) for further details. 

\begin{theorem}
Let $\bar p$ be the harmonic mean of the $p_{i}$'s as per \eqref{d} and suppose that 
\[
\bar p<N,\qquad 2<\min\{p_{1}, \dots, p_{N}\}\leq \max\{p_{1}, \dots, p_{N}\}<\bar p\, \left(1+\frac{1}{N}\right).
\]
Let $u$ be a local weak solution of \eqref{CaP} in $S_{T}$ under the growth conditions \eqref{gcond}  with 
\[
u_{0}\in L^{2}(\R^N),\qquad \emptyset\neq {\rm supp}(u_{0})\subseteq [-R_{0}, R_{0}]^{N}.
\]
Then there is a branch $\tilde u\neq 0$ of $u$ such that   
\begin{equation}
\label{lapo}
{\rm supp} (\tilde u(\cdot, t))\subseteq \prod_{j=1}^{N}[-R_{j}(t), R_{j}(t)],
\end{equation}
for any $t< T$, where 
\begin{equation}
\label{sqw}
R_{j}(t)=2\, R_{0}+Ct^{\frac{N\, (\bar p-p_{j})+\bar p}{\lambda\, p_{j}}}\|u_{0}\|_{1}^{\frac{\bar p}{p_{j}}\frac{p_{j}-2}{\lambda}},\qquad \lambda=N(\bar p -2)+\bar p.
\end{equation}
\end{theorem}

Let us make some comments on the result. 
\begin{enumerate}
\item {\em Optimality of the estimated speed.}\\
First, the obtained velocities of propagation are optimal in each direction.  Indeed, the Lebesgue measure of ${\rm supp}(\tilde u(\cdot, t))$ is estimated for large times through \eqref{lapo}  as
\[
\left|{\rm supp}(\tilde u(\cdot, t))\right|\leq C t^{\sum_{i=1}^{N}\frac{N\, (\bar p-p_{i})+\bar p}{\lambda\, p_{i}}}\|u_{0}\|_{1}^{\sum_{i=1}^{N}\frac{\bar p}{p_{i}}\frac{p_{i}-2}{\lambda}}=Ct^{\frac{N}{\lambda}}\|u_{0}\|_{1}^{\frac{N}{\lambda}(\bar p-2)}.
\]
On the other hand, we will prove in Theorem \ref{thL1Linf} below the $L^{\infty}-L^{1}$ estimate
\[
\|\tilde u(\cdot, t)\|_{\infty}\leq C\, t^{-\frac{N}{\lambda}}\|u_{0}\|_{1}^{\frac{\bar p}{\lambda}}.
\]
These two estimates imply
\[
\|\tilde u(\cdot, t)\|_{1}\leq \|\tilde u(\cdot, t)\|_{\infty}\left|{\rm supp}(u(\cdot, t))\right|\leq C\, \|u_{0}\|_{1}.
\]
Observe that if also $u_{0}\geq 0$, the branch, being compactly supported, preserves the mass along the motion. Therefore, if any of the two estimates above were asymptotically (as $t\to +\infty$) less than optimal, this would contradict mass conservation.
\item
{\em Assumptions on the $p_{i}$'s.}\\
The main assumption on the parameters is the constraint
\begin{equation}
\label{pal}
p_{\rm max}:=\max\{p_{1}, \dots, p_{N}\}<\bar p\, \left(1+\frac{1}{N}\right).
\end{equation}
which is a purely anisotropic requirement, since for $p_{i}\equiv p$ the latter is always satisfied. The explicit power of $t$ in \eqref{sqw} goes to zero as $ p_{\rm max}\to \bar p(1+1/N)$, showing that condition \eqref{pal} represents an actual threshold rather than being a technical assumption. It is possible that, for $p_{\rm max}>\bar p(1+1/N)$, zero-speed of propagation occurs in the corresponding direction (see \cite[Ch. 5.3]{AS} for some instances of such phenomenon). In this regard, local boundedness of the solutions is ensured by the condition $\max\{2, p_{\rm max}\}<\bar p(1+2/N)$, hence the feasible range is $\bar p (1+1/N)<p_{\rm max}<\bar p(1+2/N)$. 
\item {\em Tools for the proof}.\\
A major tool in deriving $L^{1}-L^{\infty}$ estimates for solutions of \eqref{CaP} is a well-known anisotropic Sobolev inequality, as proved for example in Troisi \cite{T}. Unfortunately the latter alone seems to be insufficient to prove the optimal bound \eqref{sqw} and we will  need  a variant of the anisotropic Sobolev inequality in the parabolic setting involving small powers of $u$. As the function $t\mapsto |t|^{\alpha}$ is not Lipschitz for $\alpha\in \ ]0, 1[$, the Chain Rule turns out to be a quite delicate issue, however it can still be modified to prove the needed  functional inequality in Proposition  \ref{SOBE} below (see also Theorem \ref{PAS} for its parabolic counterpart). We feel that these tools can have a wide range of applications to asymptotic and/or decay estimates for anisotropic equations.

\item {\em Further investigation.}\\
Let us also mention that when $p_{i}\equiv p>2$ (homogeneous principal part) and $u\geq 0$, the thesis of our main Theorem holds directly for $u$ and there is no need to select the branch $\tilde u$. This is due to the fact that, (despite uniqueness is not guaranteed without additional monotonicity hypotheses) there are no local non-trivial and non-negative solutions with vanishing initial datum by \cite[Proposition 18.1]{DBGVbook}. Indeed, even if the equation is nonlinear, subtracting a branch to a solution still gives a solution, possessing the aforementioned properties and thus being identically zero. We were not able to prove a similar statement in the anisotropic setting, since at present a weak Harnack inequality is still missing.  

\noindent
It is worth outlining that the optimality of disturbances' velocities, discussed in point (1) above, holds only when all  the $p_{i}$'s lie in the degenerate range $p_{i}>2$. While it is possible to obtain the same velocity bound when only {\em some} of the $p_{i}$ are greater than $2$ (again, see Theorem \ref{mainteo} in the last section), it seems less likely that for any remaining choices of the powers this bound is optimal. Indeed, it could happen that in some direction the diffusion is so fast that, in order to preserve the total mass, along the remaining ones the actual speed may be slower than what \eqref{sqw} dictates.

\noindent
Finally, we did not tackle the question of extinction in finite time in the anisotropic setting, referring to \cite{AS2} for a discussion of related results.

\end{enumerate}
\vskip4pt

{\em --Outline of the paper}\\
In section \ref{section2} we gather some useful embedding of parabolic anisotropic Sobolev spaces. In section \ref{section3} we prove two types of energy inequalities and derive some contractivity estimates. Section \ref{section4} is devoted  to {\em a-priori} $L^{\infty}$-estimates, both of global and local nature. In section \ref{section5} we first prove finite speed of propagation locally in time and space, then refine it to obtain a global result for a suitable branch. The final appendix is devoted to the Tikhonov example of non-uniqueness in the quasilinear setting.

\section{Preliminaries}\label{section2}

Given ${\bf p}:=(p_{1}, \dots p_{N})\in \R^{N}$, ${\bf p}>1$, we let
\[
\bar p:= \left(\frac 1 N \sum_{i=1}^{N}\frac{1}{p_{i}}\right)^{-1},
\]
and, assuming $\bar p<N$, we define
\[
\bar p^{*}:=\frac{N\bar p}{N-\bar p}.
\]

Given $T>0$ we set $S_{T}=\R^N\times \, ]0, T[$ and for any rectangular domain $\Omega\subseteq \R^N$, $\Omega_{T}=\Omega\times \, ]0, T[$.  
 
 We define the following spaces
\[
W^{1, {\bf p}}_{0}(\Omega):=\left\{u\in W^{1,1}_{0}(\Omega):D_iu\in L^{p_{i}}(\Omega)\right\}
\]
\[
W^{1, {\bf p}}_{{\rm loc}}(\Omega):=\left\{u\in L^{1}_{{\rm loc}}(\Omega):D_iu\in L^{p_{i}}_{{\rm loc}}(\Omega)\right\}
\]
\[
L^{{\bf p}}(0, T; W^{1, {\bf p}}_{0}(\Omega)):=\left\{ u\in L^{1}(0, T; W^{1, 1}_{0}(\Omega)): D_iu\in L^{p_{i}}(\Omega_{T})\right\}
\]
\[
L^{{\bf p}}(0, T; W^{1, {\bf p}}_{{\rm loc}}(\Omega)):=\left\{ u\in L^{1}(0, T; L^{1}_{\rm loc}(\Omega)): D_iu\in L^{p_{i}}(0, T; L^{p_{i}}_{{\rm loc}}(\Omega))\right\}
\]
\[
L^{{\bf p}}_{\rm loc}(0, T; W^{1, {\bf p}}_{{\rm loc}}(\Omega)):=\left\{ u\in L^{1}_{\rm loc}(0, T; L^{1}_{\rm loc}(\Omega)): D_iu\in L^{p_{i}}_{\rm loc}(0, T; L^{p_{i}}_{{\rm loc}}(\Omega))\right\}.
\]
Let $A$ be a measurable function obeying the growth conditions \eqref{gcond}.
By a local weak solution of 
\[
u_{t}={\rm div}(A(x, u, Du))\qquad \text{in $S_{T}$}
\]
we mean a function $u\in C^{0}_{\rm loc}(0, T; L^{2}_{{\rm loc}}(\R^N))\cap L^{{\bf p}}_{\rm loc}(0, T; W^{1, {\bf p}}_{{\rm loc}}(\R^N))$ such that for any $0<t_{1}<t_{2}<T$ and any $\varphi\in C^{\infty}_{\rm loc}(0, T; C^{\infty}_{c}(\R^N))$ it holds 
\begin{equation}
\label{ws}
\left.\int u\, \varphi\, dx\right|_{t_{1}}^{t_{2}}+\int_{t_{1}}^{t_{2}}\int -u\, \varphi_{t}+A(x, u, Du)\cdot D\varphi\, dx\, dt=0,
\end{equation}
where the integral is assumed henceforth to be on $\R^N$ when no domain is specified. By density this actually holds for any $\varphi\in W^{1, 2}_{\rm loc}(0, T; L^{2}_{\rm loc}(\R^N))\cap L^{\bf p}_{\rm loc}(0, T; W^{1, {\bf p}}_{0}(\Omega))$ for any rectangular $\Omega\Subset \R^N$.

For $T<+\infty$ and an initial data $u_{0}\in L^{2}_{\rm loc}(\R^N)$, we will consider the following Cauchy problem:
\begin{equation}
\label{DP}
\begin{cases}
u_{t}={\rm div}(A(x, u, Du))&\text{in $S_{T}$},\\
u(x, 0)=u_{0}(x),&\text{$x\in \R^{N}$},
\end{cases}
\end{equation}
where $u$ is a local solution of the equation such that, in addition, $u\in C^{0}(0, T; L^{2}_{\rm loc}(\R^N))\cap L^{\bf p}(0, T; W^{1, {\bf p}}_{\rm loc}(\R^{N}))$ with $u(\cdot, 0)=u_{0}$. In other words, for $u$ to solve \eqref{DP} we require that the initial data is assumed strongly in  $L^{2}_{\rm loc}(\R^{N})$ and that $u$ has finite energy locally in space and globally in time. When $T=+\infty$, we say that $u$ is a solution of \eqref{DP} if it is so for any $T<+\infty$.

\begin{definition}
Let $u$ be a solution of \eqref{DP}. A {\em branch} $\tilde u$ of $u$ is a solution of \eqref{DP} such that $\tilde u= u$ on ${\rm supp} (\tilde u)$.
\end{definition}

For any $u$, both $0$ and $u$ itself are obvious branches, which are the {\em trivial} ones (notice that $0$ being a branch is due to the homogeneous structure assumed for the equation). Any other branch will be called {\em proper}, or {\em non trivial}.
 
\begin{remark}
\label{branch}
Examples of solutions having proper branches can be constructed simply by solving (in the finite energy space) the Cauchy problem for the $p$-Laplacian with initial datum $\varphi$ being supported in two distant balls $B_{1}$ and $B_{2}$, nontrivially in both. If $p>2$, the finite speed of propagation provides us with a solution such that, up to some time $T>0$, have two distinct branches, namely the ones starting from the initial datum $\varphi \, \chi_{B_{i}}$, $i=1,2$. 

More interesting examples are constructed in appendix \ref{sec:appA}, where it is shown that for any $\beta>0$ the slow-diffusion parabolic $p$-Laplacian equation has nontrivial solutions supported in $\{(x, t): x\ge t^{-\beta}>0\}$. The latters clearly have  zero initial datum, are sign-changing and reproduce the Tikhonov phenomenon of non-uniquess in the quasilinear setting. 

\noindent
Let now $\varphi\in C^{\infty}_{c}(\R)$ be such that ${\rm supp}(\varphi)\subseteq \{x<0\}$ and let $\tilde{v}$ be the unique energy solution of $u_{t}=\Delta_{p}u$ having $\varphi$ as initial datum. Then, by finite speed of propagation, the function
\[
v=
\begin{cases}
u&\text{on $\{x>0\}$},\\
\tilde{v}&\text{on $\{x\le 0\}$},
\end{cases}
\]
solves the equation in $\R\times\, ]0, T]$ but, despite having compactly supported initial datum, has infinite support for any $t>0$. In this case, selecting the branch $\tilde{v}$ of $v$ is necessary in order to recover finite speed of propagation of the support.
\end{remark}

The following is a useful modification of the classical anisotropic Sobolev embedding by Troisi \cite{T}.

\begin{proposition}[Sobolev embedding]\label{SOBE}
Let $\Omega\subseteq \R^N$ be a rectangular domain, $\bar p<N$ and $\alpha_{i}>0$, $i=1, \dots, N$. Then there exists a constant $C=C(N, {\bf p}, {\bf \alpha})>0$ such that
\begin{equation}
\label{ST}
\|u\|_{L^{p^{*}_{\alpha}}(\Omega)}\leq C\prod_{i=1}^{N}\|D_i|u|^{\alpha_{i}}\|_{L^{p_{i}}(\Omega)}^{\frac{1}{\tilde{\alpha}}}
\end{equation}
for all $u\in W^{1, {\bf p}}_{0}(\Omega)$, where
\[
p_{\alpha}^{*}=\bar p^{*}\frac{\tilde{\alpha}}{N},\qquad \tilde{\alpha}=\sum_{i=1}^{N}\alpha_{i}.
\]
\end{proposition}

\begin{proof}
We can suppose that the right-hand side is finite and, by a monotone convergence argument, that $u$ and $\Omega$ are bounded.

First we claim that when $\alpha>0$, $p\geq 1$ and $u\in W^{1,1}_{0}(\Omega)\cap L^{\infty}(\Omega)$ it holds
\begin{equation}
\label{dalpha}
\int |D_{i}|u|^{\alpha}|^{p}\, dx=\alpha\, \int ||u|^{\alpha-1}D_{i} |u||^{p}\, dx,
\end{equation}
whenever the left-hand side is finite, where we set
\begin{equation}
\label{conv}
|u|^{\alpha-1}D_{i} |u|=0\quad \text{a.e. on $D_{i}u=0$}.
\end{equation}
This readily follows when $\alpha\geq 1$ from the boundedness of $u$ and the classical Chain Rule for Sobolev functions. In the case $\alpha\in \ ]0, 1[$ first observe that, since the left-hand side of \eqref{dalpha} is finite, $D_{i}|u|^{\alpha}\equiv 0\equiv  D_{i}u$ a.e. on $\{u=0\}$. Therefore it suffices to show that the weak derivatives of $|u|^{\alpha}$ and $|u|$ obey
\begin{equation}
\label{dalpha2}
D_{i}|u|^{\alpha}=\alpha \, |u|^{\alpha-1}\,  D_{i}|u|,\qquad \text{a.e. on $\{u\neq 0\}$}.
\end{equation}
By the absolute continuity on lines  characterization of Sobolev functions \cite[Theorem 10.35]{L}, we can choose a representative such that for a.e. line $l$ parallel to $e_{i}$, {\em both}\footnote{That this is true follows by suitably modifying the proof of \cite[Theorem 10.35]{L}: it suffices to start with $v_{n}\in C^{\infty}_{c}(\R^N)$ such that $D_{i}v_{n}\to D_{i}|u|^{\alpha}$ in $L^{p}(\R^N)$ and then observe that $D_{i}v_{n}^{1/\alpha}\to D_{i}|u|$ in $L^{p}(\R^N)$ as well, by the ${\rm Lip}_{\rm loc}$ character of $t\mapsto |t|^{1/\alpha}$.} the restrictions of $|u|^{\alpha}$ and $|u|$ belong to $AC(l)$, with corresponding (classical) derivatives coinciding ${\mathcal L}^{1}$ a.e. with the $i$-th weak derivative. Moreover, $|u|^{\alpha}=f(|u|)$ with $f(t):=|t|^{\alpha}\in AC_{\rm loc}(\R)$ mapping null-measure sets to  null-measure sets. Therefore \cite[Theorem 3.44]{L} applies, giving, for ${\mathcal L}^{N-1}$ a.e. $x$ such that $x_{i}=0$
\[
\frac{d}{dt} |u|^{\alpha}(x+t\, e_{i})=\alpha |u|^{\alpha-1}(x'+te_{i})\, \frac{d}{dt} |u|(x+t\, e_{i})
\]
for ${\mathcal L}^{1}$ a.e. $t$ such that $u(x+te_{i})\neq 0$. By Fubini's theorem and the absolute continuity on lines  characterization of Sobolev functions, this shows \eqref{dalpha2} and thus \eqref{dalpha}.

Let us fix from now on a representative for which $|u|^{\alpha}\in AC(l)$ for almost every line parallel to $e_{i}$, and define 
$v_{i}=|u|^{\alpha_{i}}$. By the previous discussion, the classical partial derivative and the weak derivative coincide ${\mathcal L}^{1}$ a.e. on any line, with \eqref{dalpha2} holding as long as we assume \eqref{conv}.
For $i=1, \dots, N$ denote
\[
\int_{\R_{i}}v(x)\, dt=\int_{\R}v(x+te_{i})\, dt.
\]
Consider $N$ parameters $s_{i}\geq 1$, $i=1, \dots, N$ and observe that for any $i=1, \dots, N$ and a.e. $x\in \R^N$
\[
v_{i}^{s_{i}}(x)\leq s_{i}\int_{\R_{i}}v_{i}^{(s_{i}-1)}(x)|D_{i} v_{i}|(x)\, dt\leq s_{i}\left(\int_{\R_{i}}v_{i}^{(s_{i}-1)\, p_{i}'}(x)\, dt\right)^{\frac{1}{p_{i}'}}\left(\int_{\R_{i}}|D_{i}v_{i}|^{p_{i}}(x)\, dt\right)^{\frac{1}{p_{i}}}.
\]
Therefore
\[
\prod_{i=1}^{N} v_{i}^{\frac{s_{i}}{N-1}}(x)\leq C\prod_{i=1}^{N}\left(\int_{\R_{i}}v_{i}^{(s_{i}-1)\, p_{i}'}(x)\, dt\right)^{\frac{1}{(N-1)\, p_{i}'}}\left(\int_{\R_{i}}|D_{i}v_{i}|^{p_{i}}(x)\, dt\right)^{\frac{1}{(N-1)\, p_{i}}}
\]
We successively integrate with respect to each variable $x_{j}$, observing that, at each step, the $j$-th factor of the product in the right-hand side does not depend on $x_{j}$. Regarding the other factors, it holds
\[
\sum_{i\neq j}\frac{1}{(N-1)\, p_{i}'}+\frac{1}{(N-1)\, p_{i} }=1,
\]
 so that at each step we can apply H\"older's inequality to obtain after $N$ integrations
\[
 \int \prod_{i=1}^{N} v_{i}^{\frac{s_{i}}{N-1}}\, dx\leq C\prod_{i=1}^{N}\left(\int v_{i}^{(s_{i}-1)\, p_{i}'}\, dx\right)^{\frac{1}{(N-1)\, p_{i}'}}\left(\int |D_{i}v_{i}|^{p_{i}}\, dx\right)^{\frac{1}{(N-1)\, p_{i}}}.
\]
i.e.
\[
 \int |u|^{\frac{1}{N-1}\sum_{i=1}^{N}\alpha_{i}\, s_{i}}\, dx\leq C\prod_{i=1}^{N}\left(\int |u|^{\alpha_{i}\, (s_{i}-1)\, p_{i}'}\, dx\right)^{\frac{1}{(N-1)\, p_{i}'}}\left(\int |D_{i}|u|^{\alpha_{i}}|^{p_{i}}\, dx\right)^{\frac{1}{(N-1)\, p_{i}}}.
 \]
Finally, the system in the $N+1$ variables $s_{j}, q$ 
 \[
 \alpha_{j}\, (s_{j}-1)\,  p_{j}'=q,\quad j=1,\dots, N, \qquad \frac{1}{N-1}\sum_{i=1}^{N}\alpha_{i}\, s_{i}=q
 \]
is solvable with  $q=p_{\alpha}^{*}$ (and $s_{j}\geq 1$), so that simple algebraic manipulations lead to \eqref{ST}.
\end{proof}
  
When $\alpha_{i}\equiv 1$ we let $p^{*}_{{\bf \alpha}}=\bar p^{*}$.

\begin{theorem}[Parabolic Anisotropic Sobolev embedding]
\label{PAS}
Let $\Omega\subseteq \R^N$ be a rectangular domain, $\bar p<N$, $\alpha_{i}>0$ for $i=1,\dots, N$ and $\sigma\in [1, p_{\alpha}^{*}]$. For any $\theta\in [0, \frac{\bar p}{ \bar p^{*}}]$ define 
\[
q=q(\theta, {\bf p}, {\bf \alpha})=\theta \,  p^{*}_{\alpha}+\sigma\, (1-\theta).
\]
Then there exists a constant $C=C(N, {\bf p}, {\bf \alpha}, \theta, \sigma)>0$ such that
\begin{equation}
\label{PS}
\iint_{\Omega_{T}}|u|^{q}\, dx\, dt\leq C\, T^{1-\theta\, \frac{\bar p^{*}}{\bar p}}\left(\sup_{t\in [0, T]}\int_{\Omega}|u|^{\sigma}(x, t)\, dx\right)^{1-\theta}\prod_{i=1}^{N}\left(\iint_{\Omega_{T}}|D_i|u|^{\alpha_{i}}|^{p_{i}}\, dx\, dt\right)^{\frac{\theta\,  \bar p^{*}}{N \, p_{i}}}
\end{equation} 
for any $u\in L^{1}(0, T; W^{1,1}_{0}(\Omega))$.
\end{theorem}

\begin{proof}
For a.e.  $t\in [0, T]$, apply H\"older's inequality as 
\[
\int |u|^{q}(x, t)\, dx\leq \left(\int |u|^{\sigma}(x, t)\, dx\right)^{1-\theta}\left(\int u^{p_{\alpha}^{*}}(x, t)\, dx\right)^{\theta}
\]
and through \eqref{ST} deduce
\[
\int |u|^{q}(x, t)\, dx\leq C\left(\int |u|^{\sigma}(x, t)\, dx\right)^{1-\theta}\prod_{i=1}^{N}\left(\int |D_i|u|^{\alpha_{i}}|^{p_{i}}(x, t)\, dx\right)^{\frac{ \theta\, \bar p^{*}}{N \, p_{i}}}.
\]
Finally we integrate in the $t$-variable and use H\"older's inequality in time as
\[
\int_{0}^{T}\prod_{i=0}^{N}f_{i}^{\gamma_{i}}\, dt\leq T^{1-\frac{1}{r}}\sup_{t\in [0, T]} f_{0}^{\gamma_{0}}\left(\int_{0}^{T}\prod_{i=1}^{N} f_{i}^{\gamma_{i}\, r}\, dt\right)^{\frac{1}{r}}\leq T^{1-\frac{1}{r}}\sup_{t\in [0, T]} f_{0}^{\gamma_{0}}\prod_{i=1}^{N}\left(\int_{0}^{T} f_{i}\, dt\right)^{\gamma_{i}}
\]
valid for
\[
\sum_{i=1}^{N}r\, \gamma_{i}=1, \qquad r\geq 1,\qquad \gamma_{i}>0.
\]
In the previous formula we set
\[
\gamma_{0}=1-\theta, \quad f_{0}(t)=\int |u|^{\sigma}(x, t)\, dx
\]
and for $i=1,\dots, N$
\[
\gamma_{i}=\frac{ \theta\, \bar p^{*}}{ N\, p_{i}},\quad r=\frac{\bar p}{\bar p^{*}}\frac{1}{\theta}\geq 1,\quad f_{i}(t)=\int |D_i|u|^{\alpha_{i}}|^{p_{i}}(x, t)\, dx
\]
to get the claim.
\end{proof}

\begin{remark}
In the isotropic case, the previous theorem ensures an analogous local summability estimate without the assumption that $u$ vanishes outside $\Omega$, just by adding an $L^{1}$ term to the right-hand side. Unfortunately, this is no longer true in the anisotropic setting, and counter-examples to the corresponding embeddings are known (see \cite{KK, HS} in the non-parabolic case). 
More precisely, under the previous assumptions on the parameters, it may happen that $L^{\bf p}(0, T; W^{1, {\bf p}}_{\rm loc}(\R^{N}))\cap L^{\infty}(0, T; L^{\sigma}_{\rm loc}(\R^{N}))$ fails to be contained in $L^{q}(0, T; L^{q}_{\rm loc}(\R^{N}))$ for $q$ given above.
In order to remove the zero boundary condition,  one is either forced to assume {\em a-priori} a suitable degree of summability, or to further constrain the location of the $p_{i}$'s.
 \end{remark}
 
The following result deals with the problem mentioned in the previous remark when $\alpha_{i}\equiv 1$.

\begin{theorem}\label{corPAS}
Let $1\leq p_{1}\leq \dots\leq  p_{N}$, $\bar p<N$ and for any $\sigma\in [1, \bar p^{*}]$ define the critical parabolic exponent
\[
\bar p_{\sigma}=\bar p\, \left(1+\frac{\sigma}{N}\right).
\]
Then the embedding
\begin{equation}
\label{emb0}
 L^{p_{N}}(0, T; L^{p_{N}}_{\rm loc}(\R^{N}))\cap L^{\bf p}(0, T; W^{1, {\bf p}}_{\rm loc}(\R^{N}))\cap L^{\infty}(0, T; L^{\sigma}_{\rm loc}(\R^{N}))\hookrightarrow L^{\bar p_{\sigma}}(0, T; L^{\bar p_{\sigma}}_{\rm loc}(\R^{N}))
 \end{equation}
holds true. Moreover, under the assumption
\begin{equation}
\label{assump}
\bar p_{\sigma}>p_{N}=\max\{p_{1},\dots, p_{N}\}
\end{equation}
we directly have
\[
L^{\bf p}(0, T; W^{1, {\bf p}}_{\rm loc}(\R^{N}))\cap L^{\infty}(0, T; L^{\sigma}_{\rm loc}(\R^{N}))\hookrightarrow L^{\bar p_{\sigma}}(0, T; L^{\bar p_{\sigma}}_{\rm loc}(\R^{N})).
\]
\end{theorem}

\begin{proof}
Apply the previous proposition with $\alpha_{i}\equiv 1$, $\theta=\bar p/\bar p^{*}$ to $u\, \eta$ for $\eta\in C^{\infty}_{c}(\R^{N})$, $\eta\equiv 1$ on an arbitrary cube $K\subseteq \R^{N}$. With the notations of our statement, it holds $q=\bar p_{\sigma}$ and we get
\[
\int_{0}^{T}\int_{K} |u|^{\bar p_{\sigma}}\, dx\, dt\leq \left(\sup_{t\in [0, T]}\int |u\, \eta|^{\sigma}(x, t)\, dx\right)^{1-\theta}\prod_{i=1}^{N}\left(\int_{0}^{T}\int|D_{i} (u\, \eta)|^{p_{i}}\, dx\, dt\right)^{\frac{\bar p}{N\, p_{i}}}.
\]
The first factor on the right is finite by assumption, while for the others we have
\[
\int_{0}^{T}\int|D_{i} (u\, \eta)|^{p_{i}}\, dx\, dt\le C\int_{0}^{T}\int_{K'}|D_{i} u|^{p_{i}}+|u|^{p_{i}}\, dx\, dt,
\]
where $K'$ is another cube such that ${\rm supp}(\eta)\subseteq K'$ and $C=C(\eta)$. {\hl To prove the first embedding \eqref{emb0}, it suffices to observe that, by the ordering of the $p_{i}$ and H\"older's inequality,  $L^{p_{i}}(0, T; L^{p_{i}}_{\rm loc}(\R^{N}))\hookrightarrow L^{p_{N}}(0, T; L^{p_{N}}_{\rm loc}(\R^{N}))$ for any $i=1, \dots, N$.} \\
Let us prove the second embedding under assumption \eqref{assump}. 
On vectors ${\bf q}\in \R^{N}$ we consider the component-wise partial ordering 
\[
(q_{1}, \dots, q_{N})\geq (r_{1}, \dots, r_{N})\quad \Leftrightarrow\quad q_{i}\geq r_{i}\quad \text{for all $i=1, \dots, N$}.
\]
In turn, this naturally defines a lattice structure on $\R^{N}$, with $\land$ denoting the minimum operation. By abuse of notation we will say that, for $\lambda\in \R$ it holds ${\bf q}\geq \lambda$ to mean ${\bf q}\geq (\lambda, \dots, \lambda)$.

\noindent
Define by recursion the following sequence ${\bf p}^{n}=(p_{1}^{n}, \dots, p_{N}^{n})\in \R^{N}$:
\[
\begin{cases}
{\bf p}^{1}=(p_{1}, \dots, p_{1})\\
{\bf p}^{n+1}= {\bf p}^{n}\land (q^{n}, \dots, q^{n}),\qquad q^{n}:=\frac{N+\sigma}{\sum_{1}^{N}\frac{1}{p^{n}_{i}}}.
\end{cases}
\]
 We claim that if \eqref{assump} holds, then
\begin{equation}
\label{claim4}
q^{n}\geq p_{N} \qquad \text{for sufficiently large $n$}.
\end{equation}
Let us postpone the proof of \eqref{claim4} momentarily and show how this implies $u\in L^{p_{N}}(0, T; L^{p_{N}}_{\rm loc}(\R^{N}))$.
For $n=1$ the standard isotropic parabolic Sobolev embedding ensures that  $u\in L^{q^{1}}(0, T; L^{q^{1}}_{\rm loc}(\R^{N}))$. Suppose that $u\in L^{q^{n}}(0, T; L^{q^{n}}_{\rm loc}(\R^{N}))$ for some $q_{n}\leq p_{N}$. Then, since ${\bf p}^{n+1}\leq {\bf p}$, H\"older's inequality ensures $u\in L^{{\bf p}^{n+1}}(0, T; W^{1, {\bf p}^{n+1}}_{\rm loc}(\R^{N}))$ and   the embedding \eqref{emb0} with vector ${\bf p}^{n+1}$ implies $u\in L^{q^{n+1}}(0, T; L^{q^{n+1}}_{\rm loc}(\R^{N}))$. Therefore in a finite number of steps we get $u\in L^{\bar p_{\sigma}}(0, T; L^{\bar p_{\sigma}}_{\rm loc}(\R^{N}))$ by the claim. Clearly, this process also proves the second stated embedding.
\vskip2pt

We next prove \eqref{claim4}. Since $q^{n}$ is a fixed multiple of the harmonic mean of ${\bf p}^{n}$ and $q^{1}\geq p_{1}=:q^{0}$, it follows from
\[
q^{n}\geq q^{n-1}\quad \Rightarrow\quad {\bf p}^{n+1}\geq {\bf p}^{n}\quad \Rightarrow \quad q^{n+1}\ge q^{n} 
\]
 that $\{q^{n}\}$ is non-decreasing. Suppose there exists the smallest integer $1\leq h< N$ such that $q^{n}\leq p_{h+1}$ for all $n\geq 0$, and let 
 $q=\lim_{n} q^{n}\leq p_{h+1}$. Then we infer  
 \[
 \lim_{n}q^{n}=\lim_{n}q^{n+1}\quad \Leftrightarrow\quad q=\frac{N+\sigma}{\sum_{1}^{h}\frac{1}{p_{i}}+\frac{N-h}{q}}\quad \Leftrightarrow \quad  q=r_{h},
 \]
where we defined for all $k= 1, \dots, N$
\[
r_{k}:=\frac{k+\sigma}{\sum_{1}^{k}\frac{1}{p_{i}}}.
\]
Notice that
\[
r_{k}\leq p_{k+1}\quad \Leftrightarrow \quad \frac{k+\sigma}{p_{k+1}}\leq \sum_{i=1}^{k}\frac{1}{p_{i}}
\]
so that, adding $1/p_{k+1}$ to both sides, rearranging and using the monotonicity of the $p_{i}$, we get
\[
r_{k}\leq p_{k+1}\quad \Rightarrow \quad r_{k+1}\leq p_{k+2}.
\]
Since $q=r_{h}\leq p_{h+1}$ by assumption, we eventually get by induction $r_{N-1}\leq p_{N}$, which is equivalent to say 
\[
p_{N}\geq \frac{N+\sigma}{\sum_{1}^{N}\frac{1}{p_{i}}}=\bar p\, \left(1+\frac{\sigma}{N}\right)=\bar p_{\sigma}.
\]
This contradicts \eqref{assump}, proving the claim \eqref{claim4} and the theorem.

\end{proof}

The following is a straightforward generalization, needed to deal with the anisotropic growth, of \cite[Ch. I, Lemma 4.1]{DBbook}. We omit its elementary proof.

\begin{lemma}
\label{iteration}
Let $\beta_{i}>0$ for $i=1, \dots, N$ and suppose $Z_{n}\geq 0$ satisfies
\[
Z_{n+1}\leq C\, b^{n}\, \frac{1}{N}\sum_{i=1}^{N}Z_{n}^{1+\beta_{i}}
\]
for some $C, b>1$. Then letting $\beta=\min\{\beta_{1}, \dots, \beta_{N}\}$,
\[
Z_{0}\leq \frac{1}{C^{\frac1\beta}\,  b^{\frac{1}{\beta^{2}}}}\quad \Rightarrow\quad Z_{n}\to 0.
\]
\end{lemma}

\section{Basic properties of solutions}\label{section3}

The following energy estimate can be proved through a standard Steklov averaging procedure.
\begin{lemma}[Energy Inequality]
Let $u$ be a weak solution of \eqref{ws}. There exists a constant $C=C(N, {\bf p}, \Lambda)>0$ such that for any $\eta\geq 0$ of the form
\begin{equation}
\label{test}
\eta(x, t)=\prod_{i=1}^{N}\eta_{i}^{p_{i}}(x_{i}, t), \qquad \eta_{i}\in C^{\infty}(0, T; C^{\infty}_{c}(\R))
\end{equation}
and $0<t_{1}<t_{2}<T$, $k\in \R$, it holds
\begin{equation}
\label{ei}
\begin{split}
&\left.\int (u-k)_{+}^{2}(x, t)\, \eta(x, t)\, dx\right|_{t_{1}}^{t_{2}} +\frac{1}{C}\sum_{i=1}^{N}\int_{t_{1}}^{t_{2}}\int\left|D_{i}\big((u-k)_{+}\, \eta\big)\right|^{p_{i}}\, dx\, dt\leq \\
&\qquad \quad \int_{t_{1}}^{t_{2}}\int (u-k)_{+}^{2}\, |\eta_{t}|\, dx\, dt+C\sum_{i=1}^{N}\int_{t_{1}}^{t_{2}}\int (u-k)_{+}^{p_{i}}\, |D_{i}\eta^{\frac{1}{p_{i}}}|^{p_{i}}\, dx\, dt.
\end{split}
\end{equation}
\end{lemma}
\begin{proof}
Modulo a Steklov averaging process we can suppose that $u_{t}$ exists and that \eqref{ws} holds for the test function $\varphi:=\eta \, (u-k)_{+}$. In this case \eqref{ws} becomes
\begin{equation}
\label{h1}
 \int_{t_{1}}^{t_{2}}\int u_{t}\, \varphi\, dx\, dt+\int_{t_{1}}^{t_{2}}\int A(x, u, Du)\cdot D\varphi\, dx=0.
 \end{equation}
 The first integral is estimated, for the chosen test function, as
 \begin{equation}
 \label{h2}
 \begin{split}
 \int_{t_{1}}^{t_{2}}\int u_{t}\, \varphi\, dx\, dt&= \int_{t_{1}}^{t_{2}}\int \left(\frac{(u-k)_{+}^{2}}{2}\right)_{t}\, \eta\, dx\, dt\\
 &=\left.\int  \frac{(u-k)_{+}^{2}}{2}\, \eta\, dx\right|_{t_{1}}^{t_{2}}-\int_{t_{1}}^{t_{2}}\int \frac{(u-k)_{+}^{2}}{2}\, \eta_{t}\, dx\, dt.
 \end{split}
 \end{equation}
 For the second integral, let 
 \[
 \hat\eta_{i}=\frac{\eta}{\eta_{i}^{p_{i}}},
 \]
 which does not depend on $x_{i}$, and use \eqref{gcond}  to estimate
 \[
 A(x, u, Du)\cdot D(\eta\, (u-k)_{+})\geq \Lambda^{-1}\eta \sum_{i=1}^{N}|D_{i} (u-k)_{+}|^{p_{i}}-\Lambda\, (u-k)_{+}\sum_{i=1}^{N}|D_{i}(u-k)_{+}|^{p_{i}-1}\, |D_{i}\eta|
 \]
 The last term is estimated through Young inequality as
 \begin{equation}
 \label{lj}
 \begin{split}
 (u-k)_{+}\, |D_{i}(u-k)_{+}|^{p_{i}-1}|D_{i}\eta|&=p_{i}(u-k)_{+}|D_{i}(u-k)_{+}|^{p_{i}-1}\, \eta_{i}^{p_{i}-1}\, |D_{i}\eta_{i}|\, \hat\eta_{i}\\
 &\leq \hat\eta_{i}\left(\frac{1}{2\Lambda}\eta_{i}^{p_{i}}\, |D_{i}(u-k)_{+}|^{p_{i}}+C({\bf p}, \Lambda)(u-k)_{+}^{p_{i}}\, |D_{i}\eta_{i}|^{p_{i}}\right).
 \end{split}
 \end{equation}
 Using the definition of $\hat\eta_{i}$ we get
 \begin{equation}
 \label{deta}
 \hat\eta_{i}\, \eta_{i}^{p_{i}}=\eta,\qquad \hat\eta_{i}\, |D_{i}\eta_{i}|^{p_{i}}=|D_{i}\eta^{\frac{1}{p_{i}}}|^{p_{i}},
 \end{equation}
 therefore
 \begin{equation}
 \label{adf}
 \begin{split}
 \int_{t_{1}}^{t_{2}}\int A(x, u, Du)\cdot D\varphi\, dx\geq &\frac{\Lambda^{-1}}{2}\sum_{i=1}^{N}\int_{t_{1}}^{t_{2}}\int|D_{i}(u-k)_{+}|^{p_{i}}\, \eta\, dx\, dt\\
 &- C\int_{t_{1}}^{t_{2}}\sum_{i=1}^{N}\int(u-k)_{+}^{p_{i}}\, |D_{i}\eta^{\frac{1}{p_{i}}}|^{p_{i}}\, dx\, dt.
 \end{split}
 \end{equation}
 Finally, since $p_{i}>1$ for all $i$ and $\eta_{i}\in [0, 1]$, observe that for any $v\geq 0$, 
\[
 \begin{split}
 |D_{i}(v\, \eta)|^{p_{i}}&\leq c({\bf p})\left(|D_{i}v|^{p_{i}}\, \eta^{p_{i}}+v^{p_{i}}\, \hat \eta_{i}^{p_{i}}\, \eta_{i}^{p_{i}-1}\, |D_{i}\eta_{i}|^{p_{i}}\right)\leq 
  c({\bf p})\left(|D_{i}v|^{p_{i}}\, \eta+v^{p_{i}}\, \hat \eta_{i}\, |D_{i}\eta_{i}|^{p_{i}}\right)\\
  &\leq  c({\bf p})\left(|D_{i}v|^{p_{i}}\, \eta+v^{p_{i}}\, |D_{i}\eta^{\frac{1}{p_{i}}}|^{p_{i}}\right)
  \end{split}
\]
 so that when $v=(u-k)_{+}$
 \[
 \begin{split}
 \int_{t_{1}}^{t_{2}}\int A(x, u, Du)\cdot D\varphi\, dx\geq &\frac{\Lambda^{-1}}{2c}\int_{t_{1}}^{t_{2}}\sum_{i=1}^{N}\int\left|D_{i}\big((u-k)_{+}\, \eta\big)\right|^{p_{i}}\, dx\, dt\\
 &- \frac{C}{c}\int_{t_{1}}^{t_{2}}\sum_{i=1}^{N}\int(u-k)_{+}^{p_{i}}\, |D_{i}\eta^{\frac{1}{p_{i}}}|^{p_{i}}\, dx\, dt
 \end{split}
 \]
 and the claim follows inserting this last inequality and \eqref{h2} into \eqref{h1}.
 \end{proof}

In general, we cannot always ensure that the right-hand side in \eqref{ei} is finite. This can either be required as additional condition on the solution, or it can be deduced from Theorem \ref{corPAS} whenever $\max\{p_{1},\dots, p_{N}\}<\bar p_{2}$, being a local solution {\em a priori} in $L^{\bf p}(0, T; W^{1, {\bf p}}_{\rm loc}(\R^{N}))\cap L^{\infty}(0, T; L^{2}_{\rm loc}(\R^{N}))$.

\begin{corollary}\label{corei}
Let $u\in \cap_{i=1}^{N}L^{p_{i}}(S_{T})$ solve \eqref{DP} in $S_{T}$ for $u_{0}\in L^{2}(\R^{N})$. Then $u\in L^{\infty}(0, T; L^{2}(\R^{N}))$, 
\begin{equation}
\label{l2}
\|u(\cdot, t)\|_{2}\leq \|u_{0}\|_{2}
\end{equation}
 and  for any $0<t_{1}<t_{2}<T$, $\psi\in {\rm Lip}([0, T]; \R)$, $\psi\geq 0$  and $k\in \R$ it holds
\begin{equation}
\label{ei3}
\left.\int (u-k)_{+}^{2}(x, t)\, \psi(t)\, dx\right|^{t_{2}}_{t_{1}} +\frac{1}{C}\sum_{i=1}^{N}\int_{t_{1}}^{t_{2}}\int |D_{i}(u-k)_{+} |^{p_{i}}\, \psi\,  dx\, dt\leq  \int_{t_{1}}^{t_{2}}\int (u-k)_{+}^{2}\, |\psi_{t}|\, dx\, dt.
\end{equation}
\end{corollary}

\begin{proof}
To get \eqref{l2} let $0\leq \eta\leq 1$ be as in \eqref{test}, independent on $t$ and such that 
\begin{equation}
\label{test2}
{\rm supp}(\eta_{i})\subseteq [-2R, 2R], \qquad \eta_{i}\equiv 1 \ \text{ on $[-R, R]$},\qquad  |D_{i}\eta^{\frac{1}{p_{i}}}|\leq \frac{C}{R}.
\end{equation}
With such $\eta$, apply the \eqref{ei} to both $u$ and $-u$ for $k=0$, adding the corresponding inequalities. Since $u\in C^{0}([0, T]; L^{2}_{\rm loc}(\R^{N}))$ we  can let $t_{1}\to 0$ to get
\[
\int u^{2}(x, t_{2})\, \eta\, dx\leq \int u_{0}^{2}\, dx +C\sum_{i=1}^{N}\int_{0}^{t_{2}}\int |u|^{p_{i}}\, |D_{i}\eta^{\frac{1}{p_{i}}}|^{p_{i}}\, dx\, dt.
\]
We let $R\to +\infty$, apply Fatou's lemma on the left-hand side, and observe that the right-hand side vanishes by the last property in \eqref{test2} and $u\in \cap L^{p_{i}}(S_{T})$. Therefore \eqref{l2} is proved.
\noindent
To prove \eqref{ei3}, repeat the proof of \eqref{ei} with $\varphi=\eta\, \psi$ up to \eqref{adf}, obtaining
\[
\begin{split}
&\left.\int (u-k)_{+}^{2}(x, t)\, \eta(x, t)\, \psi(t)\, dx\right|_{t_{1}}^{t_{2}} +\frac{1}{C}\sum_{i=1}^{N}\int_{t_{1}}^{t_{2}}\int\left|D_{i}\big((u-k)_{+}\big)\right|^{p_{i}}\, \eta\, \psi dx\, dt\leq \\
&\qquad \quad \int_{t_{1}}^{t_{2}}\int (u-k)_{+}^{2}\,\eta \, |\psi_{t}|\, dx\, dt+C\sum_{i=1}^{N}\int_{t_{1}}^{t_{2}}\int (u-k)_{+}^{p_{i}}\, |D_{i}\eta^{\frac{1}{p_{i}}}|^{p_{i}}\, \psi\, dx\, dt.
\end{split}
\]
Letting $R\to +\infty$ cancels the last term on the right as before, while we apply Fatou on the terms involving the spatial derivatives of $u$ and dominated convergence to the others, to obtain \eqref{ei3}. 
\end{proof}

A useful variant of the energy inequality \eqref{ei} is the following.

\begin{lemma}[Energy inequality 2]\label{Lei2}
Let $F\in C^{1,1}(\R)$ with $|F'|\leq M$ and $M\geq F''(t)>0$ for a.e. $t\in \R$ and  for suitable $M\in \R$. If $u$ is a local weak solution to \eqref{ws} and $\eta$ is of the form \eqref{test} and independent of $t$, then 
\begin{equation}
\label{F}
\begin{split}
\left.\int F(u(x, t))\, \eta(x)\, dx\right|^{t_{2}}_{t_{1}}&+\frac{1}{2\, \Lambda}\sum_{i=1}^{N}\int_{t_{1}}^{t_{2}}\int F''(u)\, \eta \, |D_{i}u|^{p_{i}}\, dx\, dt\\
&\leq C_{\Lambda}\sum_{i=1}^{N}\int_{t_{1}}^{t_{2}}\int|F'(u)|^{p_{i}}F''(u)^{1-p_{i}}|D_{i}\eta^{\frac{1}{p_{i}}}|^{p_{i}}\, dx\, dt
\end{split}
\end{equation}
and any $0<t_{1}<t_{2}< T$.
\end{lemma}

\begin{proof}
Test the equation with $\varphi= F'(u)\, \eta$, which is readily checked to be admissible since $F'\in {\rm Lip}(\R)$ is bounded and
\[
|D_{i}\varphi|\leq M|D_{i}\eta|+M|D_{i}u|\, \eta.
\]
Notice that $|F(s)|\leq M|s|$, hence $F(u(\cdot, \tau))\in L^{1}_{\rm loc}(\R^{N})$ since $u(\cdot, \tau)\in L^{2}_{\rm loc}(\R^{N})$, so by Steklov averaging we can compute
\[
\int_{t_{1}}^{t_{2}}\int u_{t}\varphi\, dx\, dt=\int_{t_{1}}^{t_{2}}\int \big(F(u)\, \eta\big)_{t}\, dx\, dt=\int F(u(x, t_{2}))\, \eta(x)\, dx-\int F(u(x, t_{1}))\, \eta(x)\, dx
\]
for any $T>t_{2}>t_{1}>0$. Therefore we have
\[
\left.\int F(u(x, t))\, \eta(x)\, dx\right|^{t_{2}}_{t_{1}}+\sum_{i=1}^{N}\int_{t_{1}}^{t_{2}}\int A_{i}(x, u, Du) \, \big(F''(u)\, \eta \, D_{i}u+ F'(u)\, D_{i}\eta\big)\, dx\, dt=0
\]
which implies by \eqref{gcond}
\begin{equation}
\label{lja}
\begin{split}
\left.\int F(u(x, t))\, \eta(x)\, dx\right|^{t_{2}}_{t_{1}}+\frac{1}{\Lambda}&\sum_{i=1}^{N}\int_{t_{1}}^{t_{2}}\int F''(u)\, \eta \, |D_{i}u|^{p_{i}}\, dx\, dt\\
&\leq \Lambda \sum_{i=1}^{N}\int_{t_{1}}^{t_{2}}\int |D_{i}u|^{p_{i}-1}  \, |F'(u)|\,  |D_{i}\eta|\, dx\, dt.
\end{split}
\end{equation}
Proceeding as in \eqref{lj} and making use of \eqref{deta} we can bound the right-hand side as
\[
\begin{split}
|D_{i}u|^{p_{i}-1}\,  |F'(u)|\, |D_{i}\eta|&=  \hat\eta_{i}\Big(|D_{i}u|^{p_{i}-1}\, F''(u)^{1-\frac{1}{p_{i}}}\, \eta_{i}^{p_{i}-1} \frac{|F'(u)|}{F''(u)^{1-\frac{1}{p_{i}}}}|D_{i}\eta_{i}|\Big)\\
& \leq \frac{1}{2\, \Lambda} F''(u)\, \eta\, |D_{i}u|^{p_{i}} + C_{\Lambda}|F'(u)|^{p_{i}}F''(u)^{1-p_{i}}|D_{i}\eta^{\frac{1}{p_{i}}}|^{p_{i}},
\end{split}
\]
which, inserted into \eqref{lja} gives the claim.
\end{proof}

\begin{corollary}
\label{corL1}
Suppose $u\in \cap_{i=1}^{N}L^{p_{i}}(S_{T})$ solves \eqref{DP} in $S_{T}$ for $u_{0}\in L^{1}(\R^{N})\cap L^{2}(\R^{N})$. Then it holds
\[
\int |u(x, t)|\, dx\leq \int |u_{0}|\, dx,\qquad \forall t\in [0, T]
\]
\end{corollary}

\begin{proof}
Let $\eta$ be as in \eqref{test2} and define, for $\alpha>0$ to be determined later,
\[
F_{\eps}(s)=\int_{0}^{s}\frac{\tau}{(|\tau|^{\alpha}+\eps)^{\frac{1}{\alpha}}}\, d\tau, 
\]
so that 
\[
F_{\eps}'(s)=\frac{s}{ (|s|^{\alpha}+\eps)^{\frac{1}{\alpha}}}, \qquad F_{\eps}''(s)=\frac{\eps}{(|s|^{\alpha}+\eps)^{\frac{1}{\alpha}+1}}>0.
\]
All the assumptions of the previous Lemma are satisfied and \eqref{F} implies, with our choice of $F_{\eps}$,
\[
\left.\int F_{\eps}(u(x, t))\, \eta(x, t)\, dx\right|^{t_{2}}_{t_{1}}\leq C\sum_{i=1}^{N}\eps^{1-p_{i}}\int_{t_{1}}^{t_{2}}\int |u|^{p_{i}}(|u|^{\alpha}+\eps)^{p_{i}-1-\frac{1}{\alpha}}|D_{i}\eta^{\frac{1}{p_{i}}}|^{p_{i}}\, dx\, dt
\]
We then choose 
\[
\alpha=(\max\{p_{1}, \dots, p_{N}\}-1)^{-1}>0,
\]
so that $p_{i}-1-\frac{1}{\alpha}\leq 0$ for all $i=1, \dots, N$ and therefore 
\[
\left.\int F_{\eps}(u(x, t))\, \eta(x, t)\, dx\right|^{t_{2}}_{t_{1}}\leq C_{\eps}\sum_{i=1}^{N}\int_{t_{1}}^{t_{2}}\int |u|^{p_{i}}|D_{i}\eta^{\frac{1}{p_{i}}}|^{p_{i}}\, dx\, dt
\]
for any $\eps\in \ ]0, 1[$ and $C_{\eps}=C(\Lambda, {\bf p}, \eps)>0$. Since $F_{\eps}$ is $1$-Lipschitz and $u(\cdot, t)\to u_{0}$ in $L^{1}_{\rm loc}(\R^{N})$, we can let $t_{1}\to 0$ in the previous estimate. By $F_{\eps}(s)\leq |s|$, $0\leq \eta\leq 1$ and $|D_{i}\eta^{\frac{1}{p_{i}}}|\leq \frac{C}{R}$, we get
\[
\int F_{\eps}(u(x, t_{2}))\, \eta(x, t_{2})\, dx\leq \int |u_{0}|\, dx +\sum_{i=1}^{N}\frac{C_{\eps}}{R^{p_{i}}}\int_{0}^{t_{2}}\int |u|^{p_{i}}\, dx\, dt,
\]
for all $\eps\in \ ]0, 1[$ and $R\geq 1$. Let first $R\to +\infty$ to cancel out the last term thanks to the hypothesis $u\in \cap L^{p_{i}}(S_{T})$ obtaining through Fatou's Lemma 
\[
\int F_{\eps}(u(x, t_{2}))\, dx\leq \int |u_{0}|\, dx \qquad \forall t_{2}\in [0, T]
\]
and since  $0\leq F_{\eps}(s)\nearrow |s|$,  we obtain the conclusion by monotone convergence.
\end{proof}

\section{$L^{\infty}$-estimates}\label{section4}

\begin{theorem}
\label{thL1Linf}
Let $p_{1}\leq \dots\leq p_{N}$, $\bar p<N$ and  $u\in \cap_{i=1}^{N}L^{p_{i}}(S_{T})$ solve \eqref{DP} for $u_{0}\in L^{1}(\R^{N})\cap L^{2}(\R^{N})$. Then:
\begin{enumerate}
\item
if  $\bar p_{2}> 2$, then $u\in L^{\infty}_{\rm loc}(0, T; L^{\infty}(\R^{N}))$ and for any $q\in [2, \bar p_{2}]$ the following estimate holds true
\begin{equation}
\label{c1}
\sup_{t\in [\theta, T]}\|u(\cdot, t)\|_{\infty}\leq \frac{C}{\theta^{\frac{N+\bar p}{\lambda_{q}}}}\left(\int_{\theta/2}^{T}\int |u|^{q}\, dx\, dt\right)^{\frac{\bar p}{\lambda_{q}}},\qquad  \lambda_{q}=N\, (\bar p-2)+\bar p\, q, \quad  \theta>0.
\end{equation}
\item
If $\bar p_{1}> 2$, the following $L^{1}-L^{\infty}$ estimate holds true for any $\tau \in \ ]0, T]$
\begin{equation}
\label{c2}
\|u(\cdot, \tau)\|_{\infty}\leq \frac{C}{\tau^{\frac{N}{\lambda}}}\|u_{0}\|_{1}^{\frac{\bar p}{\lambda}},\qquad \lambda:=\lambda_{1}=N(\bar p -2)+\bar p.
\end{equation}
\end{enumerate}
\end{theorem}

\begin{proof}
First observe that $\bar p_{1}\leq \bar p_{2}$, so that we can assume $\bar p_{2}>2$, and in that case $\lambda_{q}\geq \lambda_{2}>0$ for all $q\geq 2$. The global condition $u\in \cap_{i=1}^{N}L^{p_{i}}(S_{T})$, together with Corollary  \ref{corei} and Theorem \ref{PAS}, imply that $u\in L^{\bar p_{2}}(S_{T})\cap L^{2}(S_{T})$. Therefore, by interpolation,
\[
u\in L^{q}(S_{T}),\qquad \text{for all $q\in [\min\{2, p_{1}\}, \max\{\bar p_{2}, p_{N}\}]$}.
\]
Let $k>0$ to be determined, $T> \theta>0$ and define for any $n\geq 0$
\[
 k_{n}=k-\frac{k}{2^{n}},\qquad \theta_{n}=\theta- \frac{\theta}{2^{n+1}},  \qquad S_{n}=\R^{N}\times [\theta_{n}, T], \qquad \psi_{n}(t)=\min\left\{1, \frac{2^{n+2}}{\theta}(t-\theta_{n})_{+}\right\},
\]
so that 
\[
\text{$\psi_{n}\equiv 1$ on $[\theta_{n+1}, T]$,}\qquad |(\psi_{n})_{t}|\leq \frac{2^{n+2}}{\theta}.
\]
Since $\psi_{n}(0)=0$, the energy estimate \eqref{ei3} reads
\begin{equation}
\label{eq2}
\sup_{t\in [\theta_{n}, T]}\int(u-k_{n})_{+}^{2}\, dx + \frac{1}{C}\sum_{i=1}^{N}\iint_{S_{n}}|D_{i}(u-k_{n})_{+}|^{p_{i}}\, dx\, dt  \leq \frac{2^{n+2}}{\theta}\iint_{S_{n}} (u-k_{n})_{+}^{2}\, dx\, dt.
\end{equation}
Choose $q\in [2, \bar p_{2}]$. If  $A_{n}:=\{(x, t)\in S_{n}: u\geq k_{n}\}$, Tchebichev's inequality yields
\begin{equation}
\label{A}
\iint_{S_{n}}(u-k_{n-1})_{+}^{q}\, dx\, dt\geq (k_{n}-k_{n-1})^{q}|A_{n}|=\frac{k^{q}}{2^{n\,q}}|A_{n}|
\end{equation}
so that, by H\"older's inequality and the monotonicity of $\{S_{n}\}$
\[
\iint_{S_{n}}(u-k_{n})_{+}^{2}\, dx\, dt\leq \left(\iint_{S_{n}}(u-k_{n})_{+}^{q}\, dx\, dt\right)^{\frac{2}{q}}|A_{n}|^{1-\frac{2}{q}}\leq \frac{2^{n(q-2)}}{k^{q-2}}\iint_{S_{n-1}}(u-k_{n-1})_{+}^{q}\, dx\, dt.
\]
Therefore \eqref{eq2} becomes
\begin{equation}
\label{erp}
\sup_{t\in [\theta_{n}, T]}\int(u-k_{n})_{+}^{2}\, dx + \frac{1}{C}\sum_{i=1}^{N}\iint_{S_{n}}|D_{i}(u-k_{n})_{+}|^{p_{i}}\, dx\, dt\leq C\,  \frac{2^{n\,(q-1)}}{\theta\, k^{q-2}}\iint_{S_{n-1}}(u-k_{n-1})_{+}^{q}\, dx\, dt
\end{equation}
By H\"older inequality and \eqref{A}
\[
\begin{split}
\iint_{S_{n}}(u-k_{n})_{+}^{q}\, dx\, dt&\leq |A_{n}|^{1-\frac{q}{\bar p_{2}}}\left(\iint_{S_{n}}(u-k_{n})_{+}^{\bar p_{2}}\, dx\, dt\right)^{\frac{q}{\bar p_{2}}}\\
&\leq C\, \left(\frac{2^{n\,q}}{k^{q}}\iint_{S_{n-1}}(u-k_{n-1})_{+}^{q}\, dx\, dt\right)^{1-\frac{q}{\bar p_{2}}}\left(\iint_{S_{n}}(u-k_{n})_{+}^{\bar p_{2}}\, dx\, dt\right)^{\frac{q}{\bar p_{2}}}.
\end{split}
\]
Applying \eqref{PS} for $\alpha_{i}\equiv 1$, $\sigma=2$ and $\theta=\frac{\bar p}{\bar p^{*}}=\frac{N}{N+2}$ gives $q=\bar p_{2}$ and thus by \eqref{erp}
\[
\begin{split}
\iint_{S_{n}}(u-k_{n})_{+}^{\bar p_{2}}\, dx\, dt&\leq \left(\sup_{t\in [\theta_{n}, T]}\int(u-k_{n})_{+}^{2}(x, t)\, dx\right)^{1-\frac{\bar p}{\bar p^{*}}}\prod_{i=1}^{N}\left(\iint_{S_{n}}|D_i(u-k_{n})_{+}|^{p_{i}}\, dx\, dt\right)^{\frac{\bar p}{N\, p_{i}}}\\
&\leq C \left(\frac{2^{n\,(q-1)}}{\theta\, k^{q-2}}\iint_{S_{n-1}}(u-k_{n-1})_{+}^{q}\, dx\, dt\right)^{1+\frac{\bar p}{N}}.
\end{split}
\]
Gathering together the previous two estimates, we obtain, for suitable $b>1$,
\[
\iint_{S_{n}}(u-k_{n})_{+}^{q}\, dx\, dt\leq \frac{C}{(\theta\, k^{q-2})^{\frac{q}{\bar p_{2}}(1+\frac{\bar p}{N})}}\frac{b^{n}}{k^{q\, (1-\frac{\bar p}{\bar p_{2}})}}
\left(\iint_{S_{n-1}}(u-k_{n-1})_{+}^{q}\, dx\, dt\right)^{(1+\frac{\bar p}{N})\, \frac{q}{\bar p_{2}}+1-\frac{q}{\bar p_{2}}}.
\]
Setting
\[
X_{n}=\iint_{S_{n}}(u-k_{n})_{+}^{q}\, dx\, dt, \qquad \alpha =\frac{q}{N+2}, \qquad \beta=\frac{q}{\bar p}\, \frac{N+\bar p}{N+2},\qquad \gamma=\frac{q}{\bar p}\, \frac{N\, (\bar p-2)+\bar p\, q}{N+2}
\]
 we thus obtained the recursive inequality
\[
X_{n}\leq \frac{C}{\theta^{\beta} \, k^{\gamma}}\, b^{n}\, X_{n-1}^{1+\alpha},\qquad n\geq 1.
\]
By the classical version of Lemma \ref{iteration} with $N=1$, we have that $X_{n}\to 0$ provided $X_{0}\leq C\theta^{\frac{\beta}{\alpha}}k^{\frac{\gamma}{\alpha}}$ for a suitable constant $C$ depending only on the data. Choosing $k$ such that equality holds we therefore get
\[
\sup_{\R^N\times [\theta, T]} u\leq k= \frac{C}{\theta^{\frac{N+\bar p}{\lambda_{q}}}}\left(\int_{\theta/2}^{T}\int_{\R^N} u_{+}^{q}\, dx\, dt\right)^{\frac{\bar p}{\lambda_{q}}}\qquad \forall \theta\in \ ]0, T[
\]
implying \eqref{c1} (by considering $-u$ as well). To prove \eqref{c2}, assume $\bar p_{1}\geq 2$ and choose $q\in [2, \bar p_{1}[$.
The function $v(x, t)=u(x, t+\theta)$ still solves the equation on $S_{T-\theta}$ and is bounded there by the previous estimate. Moreover $v\in \cap L^{p_{i}}(S_{T-\theta})$ hence \eqref{c1} holds, reading
\[
\sup_{\R^N\times [\tilde\theta, T-\theta]} |v|\leq \frac{C}{\tilde\theta^{\frac{N+\bar p}{\lambda_{q}}}}\left(\int_{\frac{\tilde \theta}{2}}^{T-\theta}\int |v|^{q}\, dx\, dt\right)^{\frac{\bar p}{\lambda_{q}}}\qquad \forall \tilde\theta\in \ ]0, T-\theta[.
\]
In the latter inequality we set
\[
\tilde\theta=\theta_{n}=\frac{T-\theta}{2^{n}},\qquad M_{n}=\sup_{\R^N\times [\theta_{n}, T-\theta]}|v|
\]
 to obtain 
\[
\begin{split}
M_{n}&\leq \frac{C}{\theta_{n}^{\frac{N+\bar p}{\lambda_{q}}}}\left(\int_{\theta_{n+1}}^{T-\theta}\int |v|^{q}\, dx\, dt\right)^{\frac{\bar p}{\lambda_{q}}}\leq C\, \frac{2^{n\, \frac{N+\bar p}{\lambda_{q}}}}{(T-\theta)^{\frac{N+\bar p}{\lambda_{q}}}}\left(\int_{0}^{T-\theta}\int |v|\, dx\, dt\right)^{\frac{\bar p}{\lambda_{q}}} M_{n+1}^{(q-1)\frac{\bar p}{\lambda_{q}}}
\end{split}
\]
By the boundedness of $v$ we infer the boundedness of $\{M_{n}\}$, while
\[
(q-1)\, \frac{\bar p}{\lambda_{q}}<1\quad \Leftrightarrow\quad \bar p_{1}>2.
\]
Therefore the interpolation lemma \cite[Ch. I, Lemma  4.3]{DBbook} provides
\[
M_{0}\leq  C\left(\frac{1}{(T-\theta)^{\frac{N+\bar p}{\lambda_{q}}}}\Big(\int_{0}^{T-\theta}\int |v|\, dx\, dt\Big)^{\frac{\bar p}{\lambda_{q}}}\right)^{\frac{1}{1-(q-1)\frac{\bar p}{\lambda_{q}}}}.
\]
Writing the latter in terms of $u$ and noting that 
\[
\lambda_{q}\big(1-(q-1)\frac{\bar p}{\lambda_{q}}\big)=\lambda_{1},
\]
we obtain
\[
\sup_{\R^N}|u(\cdot, T-\theta)|\leq C\, \frac{1}{(T-\theta)^{\frac{N+\bar p}{\lambda_{1}}}}\left(\iint_{\R^N\times [\theta, T]} |u|\, dx\, dt\right)^{\frac{\bar p}{\lambda_{1}}},\qquad \forall \theta\in \ ]0, T[.
\]
Using Corollary \ref{corL1} while setting $\tau:=T-\theta\in\  ]0, T[$ finally gives \eqref{c2}.
\end{proof}

Another type of $L^{\infty}$ estimate is the following, which instead is purely local.

\begin{lemma}
Let 
\begin{equation}
\label{condp}
p_{1}\leq \dots \leq p_{N}\quad \text{and}\quad  \max\{2, p_{N}\}<\bar p_{2}.
\end{equation}
Define for $k>0$ the  functions
\begin{equation}
\label{defgh}
g(k)=\sum_{i=1}^{N}k^{\bar p_{2}-p_{i}},\qquad h(k)=\left(\sum_{i=1}^{N}k^{p_{i}-\bar p_{2}}\right)^{-1},
\end{equation}
and for $T\in \R$, $M, \lambda>0$ the intrinsic rectangle
\[
Q_{\lambda, M}= \prod_{i=1}^{N}\left[-\lambda^{\frac{1}{p_{i}}}, \lambda^{\frac{1}{p_{i}}}\right]\times [T-M\, \lambda, T].
\]
If $u$ is a weak supersolution to \eqref{ws} in $Q_{\lambda, M}$, then
\begin{equation}
\label{supest}
\|u_{+}\|_{L^{\infty}(Q_{\lambda/2, M})}\leq g^{-1}(1/M)+ h^{-1}\left(C\Big(M\, \dashiint_{Q_{\lambda, M}} u_{+}^{\bar p_{2}}\, dx\Big)^{\frac{\bar p}{N+\bar p}}\right). 
\end{equation}
\end{lemma}

\begin{proof}
Define for any $k>0$ and $\lambda, M, T$
\[
 k_{n}=k-\frac{k}{2^{n}},\quad \theta_{n}=T-\frac{M\, \lambda}{2}\left(1+ \frac{1}{2^{n}}\right), \quad r_{n, i}=\frac{\lambda^{\frac{1}{p_{i}}}}{2^{\frac{1}{p_{i}}}}\left(1+ \frac{1}{2^{n}}\right)
  \quad Q_{n}=\prod_{i=1}^{N} [-r_{n, i}, r_{n, i}]\times [\theta_{n}, T]
\]
so that $Q_{n+1}\subseteq Q_{n}\subseteq Q_{\lambda, M}$ for all $n\geq 0$ and formally $Q_{\infty}=Q_{\lambda/2, M}$. Construct functions $\eta_{n}\in C^{\infty}(Q_{n}; [0, 1])$ of the form \eqref{test} such that
\[
\left.\eta_{n}\right|_{\partial_{p}Q_{n}}\equiv 0,\qquad \left.\eta_{n}\right|_{Q_{n+1}}\equiv 1,\qquad |D_{i}\eta_{n}^{\frac{1}{p_{i}}}|\leq  \frac{C\, 2^{n}}{\lambda^{\frac{1}{p_{i}}}},\qquad  |(\eta_{n})_{t}|\leq \frac{C\, 2^{n}}{M\, \lambda}
\]
Apply \eqref{ei} to obtain
\[
\begin{split}
\sup_{t\in [\theta_{n}, T]}&\int(u-k_{n})_{+}^{2}\, \eta_{n}(x, t)\, dx + \sum_{i=1}^{N}\iint_{Q_{n}}\left|D_{i}\big((u-k_{n})_{+}\, \eta_{n}\big)\right|^{p_{i}}\, dx\, dt \\
& \leq C\, 2^{n}\left(\frac{1}{M\, \lambda}\iint_{Q_{n}} (u-k_{n})_{+}^{2}\, dx\, dt+\frac{1}{\lambda}\sum_{i=1}^{N}\iint_{Q_{n}}(u-k_{n})_{+}^{p_{i}}\, dx\, dt.\right).
\end{split}
\]
Letting $p_{0}=2$, $A_{n}=\{(x, t)\in Q_{n}: u\geq k_{n}\}$ and recalling \eqref{A}, it holds for any $i=0, \dots, N$
\[
 \iint_{Q_{n}} (u-k_{n})_{+}^{p_{i}}\, dx\, dt\leq \Big(\iint_{Q_{n}} (u-k_{n})_{+}^{\bar p_{2}}\, dx\, dt\Big)^{\frac{p_{i}}{\bar p_{2}}}|A_{n}|^{1-\frac{p_{i}}{\bar p_{2}}}\leq \frac{C b^{n}}{k^{\bar p_{2}-p_{i}}}\iint_{Q_{n-1}} (u-k_{n-1})_{+}^{\bar p_{2}}\, dx\, dt,
 \]
 for some numerical $b>1$, so that by the definition \eqref{defgh} we obtain
 \[
 \begin{split}
\sup_{t\in [\theta_{n}, T]}&\int(u-k_{n})_{+}^{2}\, \eta_{n}(x, t)\, dx + \sum_{i=1}^{N}\iint_{Q_{n}}\left|D_{i}\big((u-k_{n})_{+}\, \eta_{n}\big)\right|^{p_{i}}\, dx\, dt \\
& \leq C\, b^{n}\left(\frac{1}{M\, \lambda\, k^{\bar p_{2}-2}}+\frac{1}{\lambda\, h(k)}\right)\iint_{Q_{n-1}} (u-k_{n-1})_{+}^{\bar p_{2}}\, dx\, dt.
\end{split}
\]
As $\{k_{n}\}$ and $\{\theta_{n}\}$ are increasing and $\eta_{n}\equiv 1$ on ${\rm supp}(\eta_{n+1})$, 
\[
\begin{split}
\sup_{t\in [\theta_{n+1}, T]}\int\big((u-k_{n+1})_{+}\eta_{n+1}\big)^{2}(x, t)\, dx &\leq \sup_{t\in [\theta_{n}, T]}\int(u-k_{n})_{+}^{2}\, \eta_{n}(x, t)\, dx\\
&\leq C\, b^{n}\left(\frac{1}{M\, \lambda\, k^{\bar p_{2}-2}}+\frac{1}{\lambda\, h(k)}\right)\iint_{Q_{n-1}} (u-k_{n-1})_{+}^{\bar p_{2}}\, dx\, dt. 
\end{split}
\]
Again by the monotonicity of $\{k_{n}\}$ and $\{Q_{n}\}$ it holds
\[
\begin{split}
\sum_{i=1}^{N}\iint_{Q_{n+1}}\left|D_{i}\big((u-k_{n+1})_{+}\, \eta_{n+1}\big)\right|^{p_{i}}\, dx\, dt&\leq C\, b^{n+1}\left(\frac{1}{M\, \lambda\, k^{\bar p_{2}-2}}+\frac{1}{\lambda\, h(k)}\right)\iint_{Q_{n}} (u-k_{n})_{+}^{\bar p_{2}}\, dx\, dt\\
&\leq C\, b^{n}\left(\frac{1}{M\, \lambda\, k^{\bar p_{2}-2}}+\frac{1}{\lambda\, h(k)}\right)\iint_{Q_{n-1}} (u-k_{n-1})_{+}^{\bar p_{2}}\, dx\, dt
\end{split}
\]
Therefore, applying \eqref{PS} with $\alpha_{i}\equiv 1$, $\sigma=2$, $\theta= \bar p/\bar p^{*}$ and thus $q=\bar p_{2}$, we deduce, for some other constant $C, b\geq 1$, the recursive inequality
\[
\iint_{Q_{n+1}}(u-k_{n+1})_{+}^{\bar p_{2}}\, dx\, dt\leq C\, b^{n}\left(\frac{1}{M\, \lambda\, k^{\bar p_{2}-2}}+\frac{1}{\lambda\, h(k)}\right)^{1+\frac{\bar p}{N}}\Big(\iint_{Q_{n-1}} (u-k_{n-1})_{+}^{\bar p_{2}}\, dx\, dt\Big)^{1+\frac{\bar p}{N}}
\]
Now if $k$ is so large that
\begin{equation}
\label{condg}
\frac{1}{M\, k^{\bar p_{2}-2}}\leq \frac{1}{ h(k)},\quad \Leftrightarrow\quad k\geq g^{-1}(1/M),
\end{equation}
then the previous iterative inequality reads
\[
X_{n+1}\leq C\, b^{n}\left(\frac{1}{\lambda\, h(k)}\right)^{1+\frac{\bar p}{N}}X_{n}^{1+\frac{\bar p}{N}},\qquad n\geq 0
\]
where 
\[
X_{n}=\iint_{Q_{2 n}}(u-k_{2 n})_{+}^{\bar p_{2}}\, dx\, dt.
\]
By Lemma \ref{iteration} for $N=1$, $X_{n}\to 0$ whenever $X_{0}\leq C(\lambda\, h(k))^{\frac{N+\bar p}{\bar p}}$
and, taking account of \eqref{condg}, this in turn implies that 
\[
\sup_{Q_{\lambda/2, M}}u_{+}\leq \max\left\{ g^{-1}(1/M), h^{-1}\left(\frac{C}{\lambda}\Big(\iint_{Q_{0}} u_{+}^{\bar p_{2}}\, dx\, dt\Big)^{\frac{\bar p}{N+\bar p}}\right)\right\}.
\]
Note that $Q_{0}\subseteq Q_{\lambda, M}$ and 
\begin{equation}
\label{measQ}
|Q_{\lambda, M}|=M\, \lambda \prod_{i=1}^{N}\lambda^{\frac{1}{p_{i}}}=M\, \lambda^{\frac{N+\bar p}{\bar p}},
\end{equation}
so that being $h$ monotone increasing we obtain
\[
\sup_{Q_{\lambda/2, M}}u\leq \max\left\{g^{-1}(1/M), h^{-1}\left(C\Big(M\, \dashiint_{Q_{\lambda, M}} u_{+}^{\bar p_{2}}\, dx\, dt\Big)^{\frac{\bar p}{N+\bar p}}\right)\right\}.
\]

\end{proof}

The next  proof follows \cite{Kuusi}.

\begin{corollary}
Let \eqref{condp} hold and $u$ be a weak supersolution of \eqref{ws} in $\Omega_{T}$. Then $u$ has a lower semicontinuous representative.
\end{corollary}

\begin{proof}
Recall that all the infima and suprema are taken in the essential sense. 
For any $M\in \N$ define a metric in $\Omega_{T}$ as 
\[
{\rm dist}_{M}\big((x,t),  (x', t')\big)=\max \{M|t-t'|, |x_{1}-x_{1}'|^{\frac{1}{p_{1}}},\dots, |x_{N}-x_{N}'|^{\frac{1}{p_{N}}}\},
\]
with corresponding balls $B_{r, M}$. We will prove that there is  a ${\rm dist}_{1}$-metric essential lower-semicontinuous representative of $u$ and we start by fixing an arbitrary one, which we'll still denote by $u$.
By \eqref{measQ}, $\Omega_{T}$ with the induced metric and the Lebesgue measure is a locally doubling measure space. Therefore the set $E_{M}$ of Lebesgue points for $u$ has full measure, as well as $E=\cap_{M\in \N}E_{M}$. We can therefore suppose that for any $(x_{0}, t_{0})\in E$ and for every $M\in \N$
\[
\lim_{r\downarrow 0}\dashiint_{B_{r, M}} |u(x, t)-u(x_{0}, t_{0})|^{\bar p_{2}}\, dx\, dt=0.
\]
We claim that for any $(x_{0}, t_{0})\in E$
\begin{equation}
\label{lsceq2}
u(x_{0}, y_{0})\leq \lim_{r\to 0}\inf_{B_{r, 1}(x_{0}, t_{0})} u.
\end{equation}
Suppose by contradiction that 
\begin{equation}
\label{lsceq}
u(x_{0}, y_{0})-\inf_{B_{r, 1}(x_{0}, t_{0})} u\geq \eps>0\qquad \forall r<r_{0}
\end{equation}
and consider the solution $v=u(x_{0}, t_{0})-u$ to \eqref{ws}. Since $g(0)=h(0)=0$, $g$ and $h$ are continuous and increasing, we can choose $M>0$ such that 
\[
g^{-1}(1/M)+h^{-1}(C/M)<\eps/2,
\]
($C$ being the constant in \eqref{supest}) and, being $(x_{0}, t_{0})\in E_{M}$, choose $r(M)<r_{0}$ such that 
\[
B_{2r, M}(x_{0}, t_{0})\subseteq \Omega_{T},\qquad \dashiint_{B_{2r, M}(x_{0}, t_{0})} |u-u(x_{0}, t_{0})|^{\bar p_{2}}\, dx\, dt\leq \frac{1}{M^{2+\frac{N}{\bar p}}}.
\]
The previous Theorem (applied to $v(x-x_{0}, t)$) then assures that 
\[
\begin{split}
\sup_{B_{r,1}(x_{0}, t_{0})} u(x_{0}, t_{0})-u&\leq \sup_{B_{r, M}(x_{0}, t_{0})}u(x_{0}, t_{0})-u\\
&\leq g^{-1}(1/M)+h^{-1}\left(C\Big(M\, \dashiint_{B_{2r, M}(x_{0}, t_{0})} (u-u(x_{0}, t_{0}))_{+}^{\bar p_{2}}\, dx\, dt\Big)^{\frac{\bar p}{N+\bar p}}\right)\\
&\leq g^{-1}(1/M)+h^{-1}(C/M)<\eps/2
\end{split}
\]
contradicting \eqref{lsceq}. Finally, for  $(x_{0}, t_{0})\in \Omega_{T}\setminus E$ we modify the representative forcing the equality in \eqref{lsceq2}. 
\end{proof}

\section{Proof of the main Theorem}\label{section5}

Suppose that for a subset of $M$ indexes, $1\leq M\leq N$, it holds $p_{j_{1}}, \dots, p_{j_{M}}>2$.
We perform a permutation of the variables letting $x_{j_{1}}, \dots, x_{j_{N}}$ be the last $M$ ones and splitting $\R^{N}$ as $\R^{N-M}\times \R^{M}$ with  $\R^{N}\ni x=(x', x'')\in \R^{N-M}\times\R^{M}$.   
Therefore we can assume henceforth that 
\begin{equation}
\label{pm}
p_{j}>2,\qquad \text{for all $j=N-M+1, \dots, N$, \quad $M\geq 1$}.
\end{equation}
Furthermore, we will set, $K'_{r}=\{x'\in \R^{N-M}:|x'|\leq r\}$, $K''_{r}=\{x''\in \R^{M}:|x''|\leq r\}$ and say that 
\[
u\in L^{q}_{\rm loc}\big(\R^{M}; L^{q}(\R^{N-M}\times [0, T])\big)\quad \Leftrightarrow\quad \int_{0}^{T}\int_{\R^{N-M}\times K''_{r}}|u|^{q}\, dx'\, dx''\, dt <+\infty
\]
for any $r>0$.

\begin{lemma}
\label{BS}
Assume \eqref{pm}, $\bar p<N$ and $\bar p_{2}>\max\{p_{1},\dots, p_{N}\}$. Let $u\in \cap_{i=1}^{N}L^{p_{i}}_{\rm loc}\big(\R^{M}; L^{p_{i}}(\R^{N-M}\times [0, T])\big)$ be a weak solution of the Cauchy problem \eqref{DP} in $S_{T}$, with $u_{0}\in L^{2}(\R^N)$ and 
\[
{\rm supp}(u_{0})\subseteq \R^{N-M}\times K''_{R_{0}}.
\]
There exists $\alpha, \beta, c>0$ such that, letting,
\[
\tau(u, r, T):=c\, r^{\alpha} \left(\sum_{i=1}^{N} \int_{0}^{T}\int_{\R^{N-M}\times K''_{3r}} |u|^{p_{i}}\, dx\right)^{-\beta}, \qquad r\geq 2\, R_{0}
\]
then $u(\cdot, t)\equiv 0$ on $\R^{N-M}\times (K''_{2r}\setminus K''_{r})$ for all $t\in [0, \tau(u, r, T)]$.

\end{lemma}

\begin{proof}
Let $\tau\in \ ]0, \min\{1,T\}[$, $R\geq 1$, $r\geq 2\, R_{0}$ define
\[
r_{n}=2\, r+\frac{r}{2^{n}}, \qquad s_{n}=r-\frac{r}{2^{n+1}},\qquad E_{n}=K_{r_{n}}''\setminus K_{s_{n}}''
\]
so that $E_{n+1}''\subseteq E_{n}''$ and $\R^{N-M}\times E_{n}''\cap {\rm supp}\, u_{0}=\emptyset$. Choose $\eta'\in C^{\infty}_{c}(K'_{2\, R}; [0, 1])$ of the form \eqref{test} such that
\[
\eta'\equiv 1\ \text{on $K_{R}'$} ,\qquad  |D_{i}(\eta')^{\frac{1}{p_{i}}}|\leq \frac{C}{R},
\]
for $i=1,\dots, N-M$ and $\eta_{n}''\in C^{\infty}_{c}(E_{n}; [0, 1])$ of the form \eqref{test} and such that 
\[
\eta_{n}''\equiv 1\ \text{on $E_{n+1}$}, \qquad |D_{i}(\eta_{n}'')^{\frac{1}{p_{i}}}|\leq C\,\frac{2^n}{r}
\]
for $i=N-M+1,\dots, N$.
We apply \eqref{ei} for $k=0$, $\eta_{n}(x)=\eta'(x')\eta_{n}''(x'')$ and $t_{1}\to 0$ to both $u$ and $-u$.  Since $u\in C^{0}([0, T[; L^{2}_{{\rm loc}}(\R^N))$, and ${\rm supp} (u_{0})\cap {\rm supp}(\eta_{n})=\emptyset$, we obtain
\[
\int u^{2}(x, t)\, \eta_{n}(x)\, dx +\frac{1}{C}\sum_{i=1}^{N}\int_{0}^{t}\int|D_{i}(u\, \eta_{n})|^{p_{i}}\, dx\, dt\leq  C\sum_{i=1}^{N}\int_{0}^{t}\int|u|^{p_{i}}|D_{i}\eta_{n}^{\frac{1}{p_{i}}}|^{p_{i}}\, dx\, dt
\]
for any $t\in [0, \tau]$. Letting $R\to +\infty$ and using the assumption $u\in \cap_{i=1}^{N}L^{p_{i}}_{\rm loc}\big(\R^{M}; L^{p_{i}}(\R^{N-M}\times [0, T])\big)$ together with $|D_{i}(\eta')^{\frac{1}{p_{i}}}|\leq C/R$, all the terms for $i=1,\dots, N-M$ on the right hand side vanish. Therefore, through Fatou's Lemma on the left and dominated convergence on the remaining terms on the right, we infer
\[
\begin{split}
\int u^{2}(x, t)\, \eta_{n}''(x'')\, dx +\frac{1}{C}\sum_{i=1}^{N}\int_{0}^{t}\int|D_{i}(u\, \eta_{n}'')|^{p_{i}}\, dx\, dt&\leq  C\sum_{N-M+1}^{N}\int_{0}^{t}\int|u|^{p_{i}}|D_{i}(\eta_{n}'')^{\frac{1}{p_{i}}}|^{p_{i}}\, dx\, dt\\
&\leq C\sum_{N-M+1}^{N}\frac{2^{n\, p_{i}}}{r^{p_{i}}}\int_{0}^{t}\int_{\R^{N-M}\times E_{n}}|u|^{p_{i}}\, dx\, dt.
\end{split}
\]
By the usual monotonicity argument, the latter implies
\[
\int (u\, \eta_{n+1}'')^{2}(x, t)\, dx +\frac{1}{C}\sum_{i=1}^{N}\int_{0}^{t}\int|D_{i}(u\, \eta_{n+1}'')|^{p_{i}}\, dx\, dt\leq C\sum_{N-M+1}^{N}\frac{2^{n\, p_{i}}}{r^{p_{i}}}\int_{0}^{t}\int_{\R^{N-M}\times E_{n}}|u|^{p_{i}}\, dx\, dt.
\]
Let 
\[
X_{n}:=\sum_{N-M+1}^{N}\frac{2^{n\, p_{i}}}{r^{p_{i}}}\int_{0}^{t}\int_{\R^{M}\times E_{n}}|u|^{p_{i}}\, dx\, dt.
\]
For any $j=N-M+1, \dots, N$, apply \eqref{PS}  with 
\[
\alpha_{i}\equiv 1,\qquad \sigma=2,\qquad \theta=\theta_{j}=\frac{p_{j}-2}{\bar p^{*}-2},\qquad q=p_{j}
\]
 (notice that $2< p_{j}$ and $\bar p_{2}>p_{N}$ imply $0< \theta_{j}< \frac{\bar p}{\bar p^{*}}$ for all $j=N-M+1, \dots, N$), to obtain 
\[
\begin{split}
\int_0^{\tau}&\int_{\R^{N-M}\times E_{n+2}} |u|^{p_{j}}\, dx\, dt\leq \int_0^{\tau}\int |u\, \eta_{n+1}''|^{p_{j}}\, dx\, dt\\
&\leq C\, \tau^{1-\theta_{j}\frac{\bar p^{*}}{\bar p}}\left(\sup_{t\in [0, \tau]}\int (u\, \eta_{n+1}'')^{2}(x, t)\, dx\right)^{1-\theta_{j}}\prod_{i=1}^{N}\left(\int_{0}^{\tau}\int |D_{i}(u\, \eta_{n+1}'')|^{p_{i}}\, dx\, dt\right)^{\frac{\theta_{j}\,\bar p^{*}}{N\, p_{i}}}\\
&\leq C\, \tau^{1-\theta_{j}\frac{\bar p^{*}}{\bar p}}\, X_{n}^{1-\theta_{j}}\prod_{i=1}^{N}X_{n}^{\frac{\theta_{j}\,\bar p^{*}}{N\, p_{i}}}=C\, \tau^{1-\theta_{j}\frac{\bar p^{*}}{\bar p}}\, X_{n}^{1+\theta_{j}\frac{\bar p^{*}}{N}}.
\end{split}
\]
Let 
\[
p_{\rm max}=\max\{p_{j_{1}},\dots, p_{j_{M}}\},\qquad p_{\rm min}=\min\{p_{j_{1}}, \dots, p_{j_{M}}\}.
\]
Since for any $n\geq 0$
\[
X_{n}\leq  C\, \frac{2^{n\, p_{\rm max}}}{r^{p_{\rm min}}}\sum_{N-M+1}^{N}\int_{0}^{\tau}\int_{\R^{N-M}\times E_{n}}|u|^{p_{i}}\, dx\, dt,
\]
where  defining 
\[
Y_{n}:=\sum_{N-M+1}^{N}\int_{0}^{\tau}\int_{\R^{N-M}\times E_{3n}}|u|^{p_{i}}\, dx\, dt
\]
we obtained the recursive inequality
\[
Y_{n+1}\leq C\, \frac{2^{n\, p_{\rm max}}\tau^{1-\theta_{\rm max}\frac{\bar p^{*}}{\bar p}}}{r^{p_{\rm min}}}\, \frac{1}{M}\sum_{N-M+1}^{N}Y_{n}^{1+\theta_{j}\frac{\bar p^{*}}{N}}.
\]
Setting
 \[
 \gamma=\frac{N\,p_{\rm min}}{\theta_{\rm min} \,\bar p^{*}},\quad \delta=\frac{N}{\theta_{\rm min}}\left(\frac{1}{\bar p^{*}}-\frac{\theta_{\rm max}}{\bar p}\right),
 \]
 Lemma \ref{iteration} ensures that for some other constant $C=C(N, {\bf p}, \Lambda)$
 \[
 Y_{0}\leq C\, \frac{r^{\gamma}}{\tau^{\delta}},\quad  \Rightarrow\qquad \text{$Y_{n}\to 0$},
 \] 
 which in turn implies $u(x, t)\equiv 0$ for $x\in K''_{2r}\setminus K''_{r}$ and $t\in [0, \tau]$ and proves the claim for $\alpha=\gamma/\delta$, $\beta=1/\delta$.
  \end{proof}
It is useful to state the previous Lemma in the case when all the $p_{i}$'s are greater than $2$.

\begin{corollary}
Suppose $2<p_{1}\leq \dots\leq p_{N}<\bar p_{2}$ and let $u$ solve \eqref{DP} for $u_{0}\in L^{2}(\R^{N})$ with ${\rm supp}(u_{0})\subseteq K_{R_{0}}$.  There exists $\alpha, \beta, c>0$ such that, letting,
\[
\tau(u, r, T):=c\, r^{\alpha} \left(\sum_{i=1}^{N} \int_{0}^{T}\int_{K_{3r}} |u|^{p_{i}}\, dx\right)^{-\beta}, \qquad r\geq 2\, R_{0}
\]
then $u(\cdot, t)\equiv 0$ on $K_{2\, r}\setminus K_{r}$ for all $t\in [0, \tau(u, r, T)]$.
\end{corollary}

\begin{proof}
In this case the assumption $u\in \cap_{i=1}^{N}L^{p_{i}}_{\rm loc}\big(\R^{M}; L^{p_{i}}(\R^{N-M}\times [0, T])\big)$ reduces to $u\in \cap_{i=1}^{N}L^{p_{i}}(0, T; L^{p_{i}}_{\rm loc}(\R^{N}))$. But this follows from Theorem \ref{corPAS} and the condition $p_{N}<\bar p_{2}$.
\end{proof}

We can finally prove the main result of the paper. In the case $p_{i}>2$ for all $i=1, \dots, N$ we can choose $M=N$, so that, as in the previous proof, the condition $u\in \cap_{i=1}^{N}L^{p_{i}}_{\rm loc}\big(\R^{M}; L^{p_{i}}(\R^{N-M}\times [0, T])\big)$ reduces to $u\in \cap_{i=1}^{N}L^{p_{i}}(0, T; L^{p_{i}}_{\rm loc}(\R^{N}))$. As long as $\max\{p_{1}, \dots, p_{N}\}<\bar p_{2}$, this directly follows from the condition of being a weak solution together with Theorem \ref{corPAS}, as by H\"older's inequality
\[
L^{p_{i}}(0, T; L^{p_{i}}_{\rm loc}(\R^{N}))\subseteq L^{\infty}(0, T; L^{2}_{\rm loc}(\R^{N}))\cap L^{\bar p_{2}}(0, T; L^{\bar p_{2}}_{\rm loc}(\R^{N}))
\]
for all $i=1, \dots, N$.
\begin{theorem}\label{mainteo}
Assume \eqref{pm}, $\bar p<N$ and  
\begin{equation}
\label{condp2}
\bar p_{1}=\bar p\, \left(1+\frac{1}{N}\right)> \max\{p_{1}, \dots, p_{N}\}.
\end{equation}
Let $u\in \cap_{i=1}^{N}L^{p_{i}}_{\rm loc}\big(\R^{M}; L^{p_{i}}(\R^{N-M}\times [0, T])\big)$ be a weak solution of the Cauchy problem \eqref{DP} in $S_{T}$ with $u_{0}\in L^{2}(\R^N)\cap L^{1}(\R^{N})$ such that 
\[
\emptyset\neq {\rm supp}(u_{0})\subseteq \R^{N-M}\times K''_{R_{0}}.
\]
Then there is a branch $\tilde u\neq 0$ of $u$ such that for any $j=N-M+1, \dots, N$  
\begin{equation}
\label{fg}
{\rm supp} (\tilde u(\cdot, t))\subseteq \{|x_{j}|\leq R_{j}(t)\},
\end{equation}
for any $t\in\, ]0, T[$, where 
\[
R_{j}(t)=2\, R_{0}+Ct^{\frac{N\, (\bar p-p_{j})+\bar p}{p_{j}\, \lambda}}\|u_{0}\|_{1}^{\frac{\bar p}{p_{j}}\frac{p_{j}-2}{\lambda}},\qquad \lambda=N(\bar p -2)+\bar p.
\]
\end{theorem}

\begin{proof}
We start proving \eqref{fg} for the whole $u$ assuming that
\begin{equation}
\label{assth}
u\in \cap_{i=1}^{N}L^{p_{i}}(S_{T}).
\end{equation}
Since $2<p_{N}\leq \max\{p_{1}, \dots, p_{N}\}<\bar p_{1}$, Theorem \ref{thL1Linf}, point {\em 2)} applies, ensuring \eqref{c2}. Choose $\mu\in\  ]0, 1[$ and for any $\eps>0$ apply \eqref{F} with 
\[
F_{\eps}(s)=\int_{0}^{s} \tau\, (\tau^{2}+\eps^{2})^{\frac{\mu-1}{2}}\, d\tau,\qquad F_{\eps}''(s)=\frac{\mu \, s^{2}+\eps^{2}}{(s^{2}+\eps^{2})^{\frac{3-\mu}{2}}}>0.
\]
All the assumptions of Lemma \ref{Lei2} hold true, except the boundedness of $F'$, which however is not necessary being $u$ bounded on $S_{T}\setminus S_{t}$, $t>0$ by \eqref{c2}.
Using 
\[
\mu(s^{2}+\eps^{2})^{\frac{\mu-1}{2}}\leq  F''_{\eps}(s),\qquad |F_{\eps}'(u)|^{p_{i}}|F_{\eps}''(u)|^{1-p_{i}}\leq |u|^{p_{i}}(u^{2}+\eps^{2})^{\frac{\mu-1}{2}}
\]
for $i=1, \dots, N$, we get for all $0<t_{1}<t_{2}<T$ and $\eta$ of the form \eqref{test}
\[
\begin{split}
\left.\int F_{\eps}(u(x, t))\, \eta(x)\, dx\right|^{t_{2}}_{t_{1}}&+\frac{\mu}{2\, \Lambda}\sum_{i=1}^{N}\int_{t_{1}}^{t_{2}}\int (u^{2}+\eps^{2})^{\frac{\mu-1}{2}}\, \eta \, |D_{i}u|^{p_{i}}\, dx\, dt\\
&\leq C_{\Lambda}\sum_{i=1}^{N}\int_{t_{1}}^{t_{2}}\int|u|^{p_{i}}(u^{2}+\eps^{2})^{\frac{\mu-1}{2}}|D_{i}\eta^{\frac{1}{p_{i}}}|^{p_{i}}\, dx\, dt.
\end{split}
\]
Being $\mu\in \ ]0, 1[$, by monotone convergence on all the terms we obtain
\begin{equation}
\label{eix}
\begin{split}
\left.\int |u(x, t)|^{1+\mu}\, \eta(x)\, dx\right|^{t_{2}}_{t_{1}}&+\frac{1}{C}\sum_{i=1}^{N}\int_{t_{1}}^{t_{2}}\int  \eta\, |u|^{\mu-1} \, |D_{i}u|^{p_{i}}\, dx\, dt\\
&\leq C\sum_{i=1}^{N}\int_{t_{1}}^{t_{2}}\int|u|^{p_{i}+\mu-1}\, |D_{i}\eta^{\frac{1}{p_{i}}}|^{p_{i}}\, dx\, dt
\end{split}
\end{equation}
for some constant $C=C(\mu, \Lambda)>0$. A further monotone convergence argument shows that the previous estimate holds true for any $0\leq t_{1}\leq t_{2}\leq T$. Let $j\geq N-M+1$ and choose $\tilde\eta, \psi\in C^{\infty}_{c}(\R; [0, 1])$
\[
\eta(x)=\tilde\eta(x_{j})^{p_{j}}\prod_{i\neq j}\psi(x_{i})^{p_{i}} 
\]
with the properties 
\[
\tilde\eta\lfloor_{\{|s|\leq R_{0}\}}\equiv 0, \qquad \psi\lfloor_{\{|s|\leq R\}}\equiv 1, \qquad |\psi'|\leq C/R.
\]
With this test function, we let $t_{1}=0$, $t_{2}=T$ and $R\to +\infty$ in \eqref{eix}: all the terms on the right except the $j$-th one vanish, since \eqref{assth} holds and $u\in L^{1}(S_{T})$ by \eqref{corL1}, therefore by interpolation $u\in L^{p_{i}+\mu-1}(S_{T})$. On the left-hand side the term for $t_{1}=0$ vanishes since ${\rm supp}(u_{0})\subseteq \{|x_{j}|\leq R_{0}\}$ and on the other  we apply Fatou's lemma, to obtain
\begin{equation}
\label{eixx}
\int |u(x, T)|^{1+\mu}\, \tilde\eta\, dx+\frac{1}{C}\sum_{i=1}^{N}\int_{0}^{T}\int  \tilde\eta^{p_{j}}\, |u|^{\mu-1} \, |D_{i}u|^{p_{i}}\, dx\, dt\leq C\int_{0}^{T}\int|u|^{p_{j}+\mu-1}\, |D_{j}\tilde\eta|^{p_{j}}\, dx\, dt
\end{equation}
where we set for brevity $\tilde\eta(x)=\tilde\eta(x_{j})$. We specify further the function $\tilde\eta$ defining for  $r>2R_{0}$, $n\in \N$
\[
r_{n}=2\, r+\frac{r}{2^{n}}, \qquad s_{n}=r-\frac{r}{2^{n+1}}, \qquad E_{n}=\{x\in \R^{N}:s_{n}\leq |x_{j}|\leq r_{n}\},
\]
and for suitable $\tilde\eta_{n}\in C^{\infty}(\R; [0, 1])$,
\begin{equation}
\label{ghj}
\tilde\eta_{n}\equiv 1 \text{ on $E_{n+1}$},\qquad |\tilde\eta'_{n}|\leq C\, \frac{2^{n}}{r}, \qquad {\rm supp}(\tilde\eta_{n})\subseteq  \{s_{n}\leq |s|\leq r_{n}\}.
\end{equation}
Let finally
\[
\beta_{i}=\frac{p_{i}+\mu-1}{p_{i}}<1,\qquad \beta=\min\{\beta_{i}: j=1,\dots, N\},\qquad  \eta_{n}(x):=\tilde\eta_{n}^{1/\beta}(x_{j}).
\]
Clearly $\eta_{n}$ still satisfies \eqref{ghj}, while being $0\leq \tilde\eta_{n}\leq 1$
\[
\left|D_{i}|\eta_{n}u|^{\beta_{i}}\right|^{p_{i}}=\eta_{n}^{\beta_{i}\, p_{i}}\, \left|D_{i}|u|^{\beta_{i}}\right|^{p_{i}}\leq \beta_{i}^{p_{i}}\, \tilde\eta_{n}^{p_{i}}\, |u|^{\mu-1}\, |D_{i}u|^{p_{i}},
\]
for all $i\neq j$, where we used $\beta_{i}\geq \beta$ and \eqref{dalpha2} in the last inequality. Therefore \eqref{eixx} for $\tilde\eta=\tilde\eta_{n}$ provides 
\begin{equation}
\label{llp}
\int_{0}^{T}\int\left|D_{i}|\eta_{n}\, u|^{\beta_{i}}\right|^{p_{i}}\, dx\, dt\leq \frac{C\, 2^{n}}{r^{p_{j}}}\int_{0}^{T}\int_{E_{n-1}}|u|^{p_{j}+\mu-1}\, dx\, dt
\end{equation}
for all $i\neq j$, where we used the properties in \eqref{ghj} and the monotonicity of $E_{n}$. When $i=j$ we similarly have 
\[
\left|D_{j}|\eta_{n}\, u|^{\beta_{j}}\right|^{p_{j}}\leq \frac{C\, 2^{n}}{r^{p_{j}}}|u|^{p_{j}+\mu-1}+\beta_{j}^{p_{j}}\, \eta_{n}^{\beta_{j}\, p_{j}}\, |u|^{\mu-1}\, |D_{j}u|^{p_{j}}\leq \frac{C\, 2^{n}}{r^{p_{j}}}|u|^{p_{j}+\mu-1}+\beta_{j}^{p_{j}}\, \tilde\eta_{n}^{p_{j}}\, |u|^{\mu-1}\, |D_{j}u|^{p_{j}},
\]
giving \eqref{llp} for $i=j$ as well.
Since \eqref{eixx} also implies  
\[
\int |\tilde\eta_{n} \, u|^{1+\mu}(x, t)\, dx\leq \int_{E_{n}}\tilde\eta_{n-1}\, |u|^{1+\mu}(x, t)\, dx\leq \frac{C\, 2^{n}}{r^{p_{j}}}\int_{0}^{T}\int_{E_{n-1}}|u|^{p_{j}+\mu-1}\, dx\, dt
\]
for any $t\in \ ]0, T]$, we can apply Proposition \ref{PAS} with parameters 
\[
\sigma=1+\mu,\quad \alpha_{i}=\beta_{i}=\frac{p_{i}+\mu-1}{p_{i}},\quad  \theta=\frac{p_{j}-2}{p^{*}_{\alpha}-\mu -1},\quad q=p_{j}+\mu-1.
\]
Substitution gives
\[
\tilde\alpha=N\left(\frac{1}{\bar p'}+\frac{\mu}{\bar p}\right), \quad  \theta=(p_{j}-2)\frac{N-\bar p}{\lambda_{1+\mu}},
\]
where we recall that 
\[
\lambda_{1+\mu}=N\, (\bar p-2)+(1+\mu)\, \bar p.
\]
The necessary condition $\theta\in \ [0,  \frac{\bar p}{\bar p^{*}}]$ reads, after some algebraic manipulations, as
\[
p_{j}\leq \bar p\, \left(1+\frac{\mu+1}{N}\right)\quad \Leftrightarrow\quad \mu\geq N\, \left(\frac{p_{j}}{\bar p}-1\right)-1,
\]
and the latter quantity is always negative under assumption \eqref{condp2}.
Therefore for any $\mu\in \ ]0, 1[$ \eqref{PS} gives, through the previous estimates and some algebra
\[
\int_{0}^{T}\int |\tilde\eta_{n} \, u|^{p_{j}+\mu-1}\, dx\, dt\leq C\, 2^{n}\, T^{1-(p_{j}-2)\, \frac{N}{\lambda_{1+\mu}}}\left(\frac{2}{r^{p_{j}}}\int_{0}^{T}\int_{E_{n-1}}|u|^{p_{j}+\mu-1}\, dx\, dt\right)^{1+(p_{j}-2)\, \frac{\bar p}{\lambda_{1+\mu}}}
\]
which, being $\tilde\eta_{n}\equiv 1$ on $E_{n+1}$, implies
\[
\int_{0}^{T}\int_{E_{n+1}} |u|^{p_{j}+\mu-1}\, dx\, dt\leq C\, 2^{n}\,  T^{1-(p_{j}-2)\, \frac{N}{\lambda_{1+\mu}}}\left(\frac{2}{r^{p_{j}}}\int_{0}^{T}\int_{E_{n-1}}|u|^{p_{j}+\mu-1}\, dx\, dt\right)^{1+(p_{j}-2)\, \frac{\bar p}{\lambda_{1+\mu}}}.
\]
Applying the classical form of Lemma \ref{iteration} for $N=1$ gives that the condition
\begin{equation}
\label{condmu}
\int_{0}^{T}\int_{E_{0}} |u|^{p_{j}+\mu-1}\, dx\, dt\leq C\, r^{p_{j}\big(1+\frac{\lambda_{1+\mu}}{(p_{j}-2)\, \bar p}\big)} T^{\frac{N}{\bar p}-\frac{\lambda_{1+\mu}}{(p_{j}-2)\, \bar p}}
\end{equation}
implies 
\[
\int_{0}^{T}\int_{E_{2n}} |u|^{p_{j}+\mu-1}\, dx\, dt\to 0
\]
and hence
\begin{equation}
\label{supp2}
{\rm supp}(u(\cdot, t))\subseteq \R^N\setminus \{r\leq |x_{j}|\leq 2\, r\}\quad \forall t\in [0, T].
\end{equation}
To obtain \eqref{condmu}, we employ Corollary \ref{corL1} and Theorem \ref{thL1Linf} as follows:
\[
\begin{split}
\int_{0}^{T}\int_{E_{0}} |u|^{p_{j}+\mu-1}\, dx\, dt&\leq \int_{0}^{T}\|u(\cdot, t)\|_{1}\|u(\cdot, t)\|_{L^{\infty}}^{p_{j}+\mu-2}\, dt\\
&\leq C\|u_{0}\|_{1}\int_{0}^{T}\frac{\|u_{0}\|_{1}^{\frac{\bar p}{\lambda}\, (p_{j}+\mu-2)}}{t^{\frac{N}{\lambda}\, (p_{j}+\mu-2)}}\, dt\\
&\leq C\|u_{0}\|_{1}^{1+\frac{\bar p}{\lambda}(p_{j}+\mu-2)} T^{1-\frac{N}{\lambda}\, (p_{j}+\mu-2)}
\end{split}
\]
where we recall that $\lambda=\lambda_{1}=N\, (\bar p-2)+\bar p$ and, integrating in time at the last line, we assumed
\[
\frac{N}{\lambda}\, (p_{j}+\mu-2)<1\quad \Leftrightarrow \quad \mu<\bar p_{1}-p_{j},
\]
the latter being positive due to \eqref{condp2}.  The previous discussion shows that if $r$ and $T$ obey 
\[
\|u_{0}\|_{1}^{1+\frac{\bar p}{\lambda}(p_{j}+\mu-2)} T^{1-\frac{N}{\lambda}\, (p_{j}+\mu-2)}\leq C\, r^{p_{j}\big(1+\frac{\lambda_{1+\mu}}{(p_{j}-2)\, \bar p}\big)} T^{\frac{N}{\bar p}-\frac{\lambda_{1+\mu}}{(p_{j}-2)\, \bar p}}
\]
for some constant $C$ depending only on the data and on $\mu$, then \eqref{supp2} holds. This inequality can be rewritten through some algebra as
\[
r\geq CT^{\frac{N\, (\bar p-p_{j})+\bar p}{\lambda\, p_{j}}}\|u_{0}\|_{1}^{\frac{\bar p}{p_{j}}\frac{p_{j}-2}{\lambda}}.
\]
Thus \eqref{supp2} holds for any $r\geq 2\, R_{0}$ satisfying the previous one-sided inequality, implying \eqref{fg}.

Finally, we construct the claimed branch, removing assumption \eqref{assth} in doing so. Notice that \eqref{condp2} implies \eqref{condp}, hence Lemma \ref{BS} ensures the existence of $\tau(u, 2\, R_{0}, T)>0$ such that $u(\cdot, t)\equiv 0$ on $\R^{N-M}\times (K_{4R_{0}}''\setminus K_{2R_{0}}'')$ for all $t\in [0, \tau]$. Then we define for $t\in [0, \tau]$
\[
\tilde{u}(x, t)=
\begin{cases}
u(x, t)&\text{if $x\in \R^{N-M}\times K''_{3R_{0}}$},\\
0&\text{otherwise},
\end{cases}
\] 
which is a non-zero branch of $u$ on $S_{\tau}$ satisfying \eqref{assth}, due to 
\[
\int_{S_{\tau}}|\tilde u|^{p_{i}}\, dx\, dt=\int_{0}^{\tau}\int_{\R^{N-M}\times K''_{3R_{0}}} |u|^{p_{i}}\, dx'\, dx''\, dt<+\infty
\]
 for all $i=1,\dots, N$, by the assumption $u\in \cap_{i=1}^{N}L^{p_{i}}_{\rm loc}(\R^{M}; L^{p_{i}}(\R^{N-M}\times [0, T]))$.
 Suppose that 
\[
T^{*}:=\sup\big\{\tau>0: \text{$u$ has a branch in $\cap_{i=1}^{N}L^{p_{i}}(S_{\tau})$}\big\}<T,
\]
and let 
\[
\bar R= \max\{R_{j}(T):j=N-M+1,\dots, N\},\qquad   \eps=\frac{1}{2}\tau(u,  2\, \bar R, T) 
\]
where $\tau$ is given in Lemma \ref{BS}. Then, by \eqref{fg}, on the whole $S_{T^{*}-\eps}$, $u$ has a branch supported in $\R^{N-M}\times K''_{\bar R}$. Applying Lemma \ref{BS} as in the first step, $\tilde u$ can be extended up to $S_{T^{*}+\eps}$ staying in $\cap_{i=1}^{N}L^{p_{i}}(S_{T^{*}+\eps})$, contradicting $T^{*}<T$. Thus $T^{*}=T$, there exists a branch $\tilde u\in \cap_{i=1}^{N}L^{p_{i}}(S_{\tau})$, which therefore satisfies \eqref{fg}.
\end{proof}

\appendix

\section{An example \'a la Tikhonov}

\label{sec:appA}
\noindent
In this appendix we will construct a nontrivial solution to the Cauchy problem
\[
\begin{cases}
u_{t}={\rm div}(|\nabla u|^{p-2}\nabla u)&\text{in $\R^{N}\times \, ]0, +\infty[$},\\
u(\cdot, t)\to  0&\text{as $t\to 0$ in $L^{2}_{{\rm loc}}(\R^{N})$},
\end{cases}
\]
in the slow diffusion case $p>2$.
Suppose that $\alpha, \beta>0$ satisfy
\begin{equation}
\label{alphabeta}
(p-2)\alpha=1+p\beta.
\end{equation}
Then the equation $u_{t}=\Delta_{p}u$ in $\R\times\, ]0, +\infty[$ with the ansatz $u(x, t)=t^{-\alpha}U(x\, t^{\beta})$ reduces to 
\[
-\alpha\, U(s)+\beta\, U'(s)\, s=(|U'(s)|^{p-2}U'(s))',\qquad s=x\, t^{\beta}
\]
and letting $V=|U'|^{p-2}U'$, the latter can be rewritten as the non-autonomous ODE system 
\begin{equation}
\label{sys}
\begin{cases}
U'=|V|^{\frac{2-p}{p-1}}V\\
V'=-\alpha \, U+\beta \, |V|^{\frac{2-p}{p-1}}V\, s.
\end{cases}
\end{equation}
Our aim is to construct a nontrivial solution for $s\in [1, +\infty[$ to this system with initial condition  $U(1)=V(1)=0$. Notice that $U\equiv V\equiv 0$ is certainly a solution, but the right-hand side of \eqref{sys} is only H\"older continuous in $V$ and the standard Picard uniqueness cannot be applied. We also observe that a major r\^{o}le in the following construction is some {\em anti-dissipative} feature of the system. Indeed, it is well known that, even if the ODE system ${\bf x}'={\bf F}(s, {\bf x})$ has only continuous right hand side, the dissipativity condition 
\[
\big({\bf F}(s, {\bf x})-{\bf F}(s, {\bf y})\big)\cdot ({\bf x}-{\bf y})\leq 0
\]
is enough to ensure uniqueness of the Cauchy problem. Luckily, system \eqref{sys} satisfies, for ${\bf x}=(0, V)$ and ${\bf y}=(0, 0)$ the opposite inequality ${\bf F}(s, 0, V)=\beta \, |V|^{p}\, s>0$ near $s=1$.

\medskip

To construct the solution, first observe that it is enough to do it locally, since the right hand side of \eqref{sys} has sublinear growth and by Gronwall's lemma any local solution can be extended to a global one.
A local (forward) solution of \eqref{sys} is a fixed point of the operator
\[
T({\bf x})(s)=\int_{1}^{s}{\bf F}(s, {\bf x}(s))\, ds,
\]
on the space
\[
X_{\delta}=\left\{{\bf x}\in C([1, 1+\delta], \R^{2}):{\bf x}(1)=(0, 0)\right\}
\]
 with the standard uniform norm denoted by $\|\ \|_{\delta}$, where 
\[
{\bf x}=(x_{1}, x_{2}),\qquad {\bf F}(s, {\bf x})=\big(|x_{2}|^{\frac{2-p}{p-1}}x_{2}, \, -\alpha \, x_{1}+\beta \, |x_{2}|^{\frac{2-p}{p-1}}x_{2}\, s\big).
\]
 By the standard proof of the Peano existence theorem, there exists $\bar \delta\in \, ]0, 1[$ and $M, L>0$ such that, given any $\bar x$ and $\delta\in \, ]0, \bar \delta[$ with $\|\bar x\|_{\delta}\leq M$, the sequence
\[
{\bf x}_{0}=\bar {\bf x},\qquad {\bf x}_{n+1}=T({\bf x}_{n})
\]
satisifies $\|{\bf x}_{n}\|_{\delta}\le M$ for all $n\ge 0$, is $L$ - equilipschitz for $n\ge 1$ and converges in $X_{\delta}$  to a fixed point of $T$. 

We employ a sub-supersolution method with a nontrivial subsolution. The heuristic is that our Cauchy problem for  \eqref{sys} is, near $s=1$, asymptotically equivalent to 
\[
\begin{cases}
x_{1}'=|x_{2}|^{\frac{2-p}{p-1}}x_{2},\\
x_{2}'=\beta \, |x_{2}|^{\frac{2-p}{p-1}}x_{2},
\end{cases}
\qquad x_{1}(1)=x_{2}(1)=0,
\]
whose maximal nontrivial solution is, up to multiplicative constants, $x_{2}(s)\simeq x_{1}(s)\simeq  (s-1)^{\frac{p-1}{p-2}}$.\\
For $0<a\le b$, $\delta\in \, ]0, \bar \delta[$, consider the closed convex set 
\[
C_{\delta, a, b}:=\left\{{\bf x}\in X_{\delta}: 0\le x_{1}(s)\le  b\, (s-1)^{\frac{p-1}{p-2}},\ a\, (s-1)^{\frac{p-1}{p-2}}\le x_{2}(s)\le b\, (s-1)^{\frac{p-1}{p-2}},\  s\in [1, 1+\delta]\right\}.
\]
We claim that $C_{\delta, a, b}$ is $T$-invariant for suitable choices of the parameters. If $T_{i}(x_{1}, x_{2})$ is the $i$-th coordinate of $T(x_{1}, x_{2})$, $i=1,2$, given ${\bf x}=(x_{1}, x_{2})\in C_{\delta}$ we estimate for any $s\in \, ]1, 1+\delta[$
\begin{align}
T_{2}(x_{1}, x_{2})(s)&=\int_{1}^{s}-\alpha x_{1}(\tau)+\beta \, x_{2}^{\frac{1}{p-1}}(\tau)\,\tau\, d\tau \notag \\ 
&\ge - \alpha\, b\int_{1}^{s}(\tau-1)^{\frac{p-1}{p-2}}\, d\tau+\beta\, a^{\frac{1}{p-1}}\int_{1}^{s}(\tau-1)^{\frac{1}{p-2}}\, d\tau \notag\\ 
&\ge -\alpha\, b\, \frac{p-2}{2p-3}\, (s-1)^{\frac{2p-3}{p-2}}+\beta\, a^{\frac{1}{p-1}}\, \frac{p-2}{p-1}\, (s-1)^{\frac{p-1}{p-2}}; \label{a1} \\[3pt]
T_{2}(x_{1}, x_{2})(s)&\le \beta \int_{1}^{s} x_{2}^{\frac{1}{p-1}}(\tau)\, \tau\le (1+\delta)\, \beta\, b^{\frac{1}{p-1}}\, \frac{p-2}{p-1}\, (s-1)^{\frac{p-1}{p-2}};\label{a2}\\[3pt]
0\le T_{1}(x_{1}, x_{2})(s)&=\int_{1}^{s}x_{2}^{\frac{1}{p-1}}(\tau)\, d\tau\le b^{\frac{1}{p-1}}\int_{1}^{s}(\tau-1)^{\frac{1}{p-2}}\, d\tau\le b^{\frac{1}{p-1}}\, \frac{p-2}{p-1}\, (s-1)^{\frac{p-1}{p-2}}.\label{a3}
\end{align}
Using $\delta\le \bar\delta\le  1$ in the upper bound \eqref{a1} for $T_{2}$, we first define $b$ as per
\[
\frac{p-1}{p-2}\, b^{\frac{p-2}{p-1}}= \max\{1, 2\, \beta\},
\]
which, inserted into \eqref{a2} and \eqref{a3} gives through some algebra 
\begin{equation}
\label{ubapp}
 T_{2}(x_{1}, x_{2})(s)\le  b \, (s-1)^{\frac{p-1}{p-2}}\qquad  T_{1}(x_{1}, x_{2})(s)\le b\, (s-1)^{\frac{p-1}{p-2}}.
 \end{equation}
 Then we choose $\delta\in\,  ]0, \bar\delta[$ such that 
 \[
 -\alpha\, b\, \frac{p-2}{2p-3}\, (s-1)^{\frac{2p-3}{p-2}}+\beta\, a^{\frac{1}{p-1}}\, \frac{p-2}{p-1}\, (s-1)^{\frac{p-1}{p-2}}\ge \frac{\beta}{2}\, a^{\frac{1}{p-1}}\, \frac{p-2}{p-1}\, (s-1)^{\frac{p-1}{p-2}}
 \]
 (this is possible because from $p>2$ we infer $(2p-3)/(p-2)>(p-1)/(p-2)$).
 Finally, we set 
 \[
 \frac{p-1}{p-2}\, a^{\frac{p-2}{p-1}}=\frac \beta 2,
 \]
 (notice that $a\le b$), to deduce from \eqref{a2} the lower bound
 \[
 T_{2}(x_{1}, x_{2})(s)\ge a\, (s-1)^{\frac{p-1}{p-2}}.
 \]
 The latter inequality, together with \eqref{ubapp} and the trivial one $T_{1}(x_{1}, x_{2})\ge 0$
concludes the proof of the claimed $T$-invariance of $C_{\delta, a, b}$. The starting point 
 \[
 \bar{\bf  x}(s)=\big(0, a\, (s-1)^{\frac{p-1}{p-2}}\big)
 \]
 satisfies $\|\bar{\bf x}\|\le M$ by eventually further reducing $\delta$, and thus produces by iteration a nontrivial solution $(U, V)$ to \eqref{sys} with initial conditions $U(1)=V(1)=0$.

As remarked before, the solution can be extended to the whole $[1, +\infty[$, and it remains to observe that ${\rm supp}(U)=[1,+\infty[$, due to the inequality
\[
(\alpha\, U^{2}+ \frac{p-1}{p}\, |V|^{\frac{p}{p-1}})'=\beta\, |V|^{\frac{2}{p-1}}\, s\ge 0.
\] 

For any $\alpha, \beta>0$ obeying \eqref{alphabeta}, such a trajectory defines a solution to $u_{t}=\Delta_{p}u$ in $\R^{N}\times \, ]0, +\infty[$ through
\[
u(x, t)=
\begin{cases}
t^{-\alpha}U(x_{1}\, t^{\beta})&\text{if $x_{1}\, t^{\beta}\ge 1$}\\
0&\text{otherwise},
\end{cases}
\]
which has the claimed properties, since its support is $\{x_{1}\ge t^{-\beta}>0\}$.

\vskip4pt
\noindent
{\small {\bf Acknowledgement.} We thank J. L. V\'azquez (U. Aut\'onoma de Madrid) for a discussion on non-uniqueness phenomena for the Cauchy problem. S. Mosconi and V. Vespri  are members of GNAMPA (Gruppo Nazionale per l'Analisi Matematica, la Probabilit\`a e le loro Applicazioni) of INdAM (Istituto Nazionale di Alta Matematica). F. G. D\"uzg\"un is partially funded by Hacettepe University BAP through project FBI-2017-16260; S. Mosconi is partially funded by the grant PdR 2016-2018 - linea di intervento 2: ``Metodi Variazionali ed Equazioni Differenziali'' of the University of Catania.}


\begin{thebibliography}{99}
\bibitem{AT}
{\sc D. Andreucci, S. F. Tedeev}, Finite speed of propagation and other higher-order parabolic equations with general nonlinearity. {\em Interf. Free Bound.} {\bf 3} (2001), 233--264.

\bibitem{AS}
{\sc S. Antontsev, S. Shmarev}, Evolution PDEs with Nonstandard Growth Conditions. Atlantis Studies in Differential Equations {\bf 4}, Atlantis Press, Paris, 2015.

\bibitem{AS2}
{\sc S. Antontsev, S. Shmarev}, Localization of solutions of anisotropic parabolic equations. {\em Nonlinear Anal.} {\bf 71} (2009), 725--737.

\bibitem{A}
{\sc D. G. Aronson}, Widder's inversion theorem and the initial distribution problem. {\em SIAM J. Math. Anal.} {\bf 12} (1981), 639--651.

\bibitem{B}
{\sc F. Bernis}, Qualitative properties for some nonlinear higher order degenerate parabolic equations. {\em Houston J. Math.} {\bf 14} (1988), 319--352.

\bibitem{BV}
{\sc M. F. Bidaut-V\'eron}, Self-similar solutions of the $p$-Laplace heat equation: the case when $p > 2$. {\em Proc. Roy. Soc. Edinburgh Sect. A} {\bf 139} (2009), 1--43.

\bibitem{BMS}
{\sc L. Boccardo, P. Marcellini, C. Sbordone}, $L^{\infty}$-regularity for variational problems with sharp non-standard growth conditions. {\em Boll. Unione Mat. Ital.} VII. Ser. A {\bf 4} (1990), 219--225.

\bibitem{BRVV}
{\sc V. B\"ogelein, F. Ragnedda, S. Vernier Piro, V. Vespri}, 
Moser--Nash kernel estimates for degenerate parabolic equations. {\em J. Func. Anal.} {\bf 272} (2017), 2956--2986.

\bibitem{DT07}
{\sc S. P. Degtyarev, A. F. Tedeev}, $L^{1}-L^{\infty}$ estimates of the solution of the Cauchy problem for an anisotropic degenerate parabolic equation with double nonlinearity and growing initial data. {\em Mat. Sb.} {\bf 198} (2007), 46--66.

\bibitem{DT12}
{\sc S. P. Degtyarev, A. F. Tedeev}, On the solvability of the Cauchy problem with growing initial data for a class of anisotropic parabolic equations. {\em J. Math. Sci.}  {\bf 181} (2012), 28--46.

\bibitem{DBbook}
{\sc E. DiBenedetto}, Degenerate Parabolic Equations. {\em Universitext}, Springer-Verlag New York, 1993.

\bibitem{DBGVbook}
{\sc E. DiBenedetto, U. Gianazza and V. Vespri}, Harnack's inequality for degenerate and singular parabolic
equations. {\em Springer Monographs in Mathematics}, Springer-Verlag New York, 2012.

\bibitem{DBGV16}
{\sc E. DiBenedetto, U. Gianazza and V. Vespri}, Remarks on Local Boundedness and Local Holder Continuity of Local Weak Solutions to Anisotropic $p$-Laplacian Type Equations. {\em J. Elliptic Parabol. Equ.} {\bf 2} (2016) 157--169. 

\bibitem{DBH}
{\sc E. DiBenedetto, M. A. Herrero}, On the Cauchy problem and initial traces for a degenerate parabolic equation. {\em Trans. AMS} {\bf 314} (1989), 187--224.

\bibitem{DBH2}
{\sc E. DiBenedetto, M. A. Herrero}, Non negative solutions of the evolution $p$-Laplacian equation. Initial traces and Cauchy problem when $1<p<2$. {\em Arch. Rational Mech. Anal.} {\bf 111} (1990), 225--290.

\bibitem{EMM}
{\sc M. Eleuteri, P. Marcellini, E. Mascolo},
Regularity for scalar integrals without structure conditions. To appear in {\em Adv. Calc. Var.}, DOI: 10.1515/acv-2017-0037.

\bibitem{FS}
{\sc N. Fusco, C. Sbordone}, Local boundedness of minimizers in a limit case. {\em Manuscr. Math.} {\bf 69} (1990), 19--25.

\bibitem{G}
{\sc M. Giaquinta}, Growth conditions and regularity. A counterexample. {\em Manuscr. Math.} {\bf 59} (1987), 245--248.

\bibitem{HS}
{\sc J. Ha\v{s}kovec, C. Schmeiser}, A note on the anisotropic generalizations of the Sobolev and
Morrey embedding theorems. {\em Monatsh. Math.}, {\bf 158} (2009),  71--79.

\bibitem{H}
{\sc M.-C. Hong}, Existence and partial regularity in the calculus of variations. {\em Ann. Mat. Pura Appl.},  {\bf 149} (1987), 311--328.

\bibitem{Hu}
{\sc J. Hulshof}, Similarity solutions for the porous medium equation with sign changes. {\em J. Math. Anal. Appl.} {\bf 157} (1991), 75–111.

\bibitem{V2}
{\sc R. G. Iagar, A. S\'anchez, J. L. V\'azquez},
Radial equivalence for the two basic nonlinear degenerate diffusion equations.
{\em J. Math. Pures Appl.}, {\bf 89} (2008), 1--24.

\bibitem{KK}
{\sc S. N. Kruzhkov, \={I}. M. Kolod\={i}i}, On the theory of embedding of anisotropic Sobolev spaces. {\em Uspekhi Mat. Nauk}, {\bf 38}, (1983), 207--208 (in Russian). English transl.: {\em Russian Math. Surveys} {\bf 38} (1983), 188--189.

\bibitem{Kuusi}
{\sc T. Kuusi}, Lower semicontinuity of weak supersolutions to nonlinear parabolic equations. {\em Differential Integral Equations} {\bf 22} (2009), 1211--1222.

\bibitem{L}
{\sc J. L. Lions}, Quelques m\'ethodes de r\'esolution des probl\`emes aux limites non
lin\'eaires. Dunod, Paris, 1969.

\bibitem{M}
{\sc P. Marcellini}, Un esemple de solution discontinue d'un probl\`eme variationnel dans le case scalaire. Ist. Mat. ``U. Dini'' No. 11, Firenze, 1987. (In Italian).

\bibitem{Min}
{\sc G. Mingione}, Regularity of minima: an invitation to the Dark Side of the Calculus of Variations. {\em Appl. Math.} {\bf 51} (2006) 355--426.
 
\bibitem{TV}
{\sc  A. F. Tedeev, V. Vespri}, Optimal behavior of the support of the solutions to a class of degenerate parabolic systems. {\em Interfaces Free Bound.} {\bf 17} (2015), 143--156.
 
\bibitem{T}
{\sc M. Troisi}, Teoremi di inclusione per spazi di Sobolev non isotropi, {\em Ricerche Mat.} {\bf 18} (1969), 3--24. (In Italian).

\bibitem{Ty}
{\sc A. N. Tychonov}, Th\'eor\`emes d'unicit\'e pour l'\'equation de la chaleur. {\em Mat. Sb.} {\bf 42} (1935), 199--216.

 \bibitem{MX}
 {\sc M. Yu, X. Lian}, Boundedness of solutions of parabolic equations with anisotropic growth conditions. {\em Can. J. Math.}  {\bf 49} (1997), 798--809.

\bibitem{V}
{\sc J. L. V\'azquez}, New selfsimilar solutions of the porous medium equation and the theory of solutions with changing sign. {\em Nonlinear Analysis} {\bf 15} (1990), 931--942.


\bibitem{W}
{\sc  D. V. Widder}, Positive temperatures in an infinite rod. {\em Trans. Amer. Math. Soc.} {\bf 55} (1944), 85--95.
\end{thebibliography}
\end{document}